\documentclass{msc}
\usepackage{amsmath}
\usepackage{graphicx}

\usepackage{amsfonts}
\usepackage[all,2cell]{xy}
\UseAllTwocells
\xyoption{v2}
\NoResizing
\UseResizing

\theoremstyle{remarkstyle}
\newtheorem{example}[therm]{Examples}

\newcommand{\alg}[1]{\operatorname{#1 -\mathbf{Alg}}}

\def\eval{\mathrm{ev}}
\def\lang{\mathcal{L}}
\newcommand{\bemph}[1]{\textbf{\emph{#1}}}
\newcommand{\theory}[1]{\mathcal{#1}}

\newcommand{\ltil}[1]{\widetilde{#1}}

\newcommand{\ovln}[1]{\overline{#1}}

\newcommand{\posdoctrine}[2]{\xymatrix{#2 \colon #1^{\op}  \ar[r] & \pos }}

\newcommand{\doctrine}[2]{\xymatrix{#2 \colon #1^{\op}  \ar[r] & \infsl }}

\def\mT{\mathrm{T}}

\newcommand{\compun}[1]{{#1}^{\universal}}
\newcommand{\compex}[1]{{#1}^{\existential}}

\newcommand{\angbr}[2]{\langle #1,#2 \rangle}

\newcommand{\freccia}[3]{\xymatrix{#2 \colon #1  \ar[r] &  #3}}
\newcommand{\cover}[3]{\xymatrix{#2 \colon #1  \ar@{-|>}[r] &  #3}}
\newcommand{\frecciasopra}[3]{\xymatrix{ #1  \ar[r]^{#2} &  #3}}

\newcommand{\duefreccia}[3]{\xymatrix@C=0.5cm{#2 \colon #1  \ar@{=>}[r] &  #3}}
\newcommand{\modificazione}[3]{\xymatrix@C=0.5cm{#2 \colon #1  \ar@{~>}[r] &  #3}}

\newcommand{\duemorfismo}[6]{\xymatrix{
#1^{\op} \ar[rrd]^#2_{}="a" \ar[dd]_{#3^{\op}}\\
&& \pos\\
#5^{\op}  \ar[rru]_#6^{}="b"
\ar_{#4}  "a";"b"}}
\newcommand{\comsquare}[8]{ \xymatrix@+1pc{ 
#1 \ar[r]^{#5} \ar[d]_{#6} & #2 \ar[d]^{#7} \\
#3 \ar[r]_{#8} & #4 
}}
\newcommand{\pullback}[8]{ \xymatrix@+2pc{ 
#1 \pullbackcorner \ar[r]^{#5} \ar[d]_{#6} & #2 \ar[d]^{#7} \\
#3 \ar[r]_{#8} & #4 
}}
\newcommand{\quadratocomm}[8]{ \xymatrix@+1pc{ 
#1 \ar[r]^{#5} \ar[d]_{#6} & #2 \ar[d]^{#7} \\
#3 \ar[r]_{#8} & #4 
}}
\newcommand{\comsquarelargo}[8]{ \xymatrix@+1pc{ 
#1 \ar[rr]^{#5} \ar[d]_{#6} && #2 \ar[d]^{#7} \\
#3 \ar[rr]_{#8} && #4 
}}
\newcommand{\parallelmorphisms}[4]{\xymatrix@+1pc{
#1 \ar @<+4pt>[r]^{#2} \ar @<-4pt>[r]_{#3} & #4
}}
\newcommand{\relation}[4]{\xymatrix@+1pc{
\angbr{#2}{#3}\colon #1 \ar @<+4pt>[r] \ar @<-4pt>[r] & #4
}}
\newcommand{\frecceparalleleopposte}[4]{\xymatrix@+1pc{
#1 \ar@<+4pt>[r]^{#2} \ar@<-4pt>@{<-}[r]_{#3} & #4
}}
\newcommand{\equalizer}[6]{\xymatrix@+1pc{
#1 \ar[r]^{#2} & #3 \ar @<+4pt>[r]^{#4} \ar @<-4pt>[r]_{#5} & #6
}}
\newcommand{\coequalizer}[6]{\xymatrix@+1pc{
 #1 \ar @<+4pt>[r]^{#2} \ar @<-4pt>[r]_{#3} & #4 \ar[r]^{#5} & #6
}}
\newcommand{\sottoggetto}[2]{\xymatrix{
#1 \ar@{>->}[r] & #2
}}
\newcommand{\subobject}[3]{\xymatrix{
#1 \ar@{>->}[r]^{#2} & #3
}}

\newcommand{\pullbackcorner}[1][ul]{\save*!/#1+1.2pc/#1:(1,-1)@^{|-}\restore}

\def\ED{\operatorname{\mathbf{ExD}}}            
\def\PD{\operatorname{\mathbf{SD}}}             

\def\SD{\operatorname{\mathbf{SD}}}    
\newcommand{\exponent}[1]{#1_{exp}} 
\def\UD{\operatorname{\mathbf{UnD}}}  

\def\mC{\mathcal{C}}

\def\mD{\mathcal{D}}

\def\x{\times}

\def\Ainv{\hbox{\rotatebox[origin=c]{180}{A}}}
\newbox\erove \setbox\erove=\hbox{\reflectbox{E}}
\def\Einv{\usebox\erove}
\def\pr{\operatorname{ pr}}

\def\id{\operatorname{ id}}

\def\op{\operatorname{ op}}

\def\universal{\operatorname{un}}
\def\existential{\operatorname{ex}}

\def\Sub{\operatorname{ Sub}}

\def\Set{\operatorname{\mathbf{Set}}}

\def\:{\colon}
\def\infsl{\operatorname{\mathbf{InfSL}}}

\def\pos{\operatorname{\mathbf{Pos}}}
\def\ev{\operatorname{ev}}

\newcommand{\coprodotto}[2]{{\scriptsize \begin{pmatrix}
#1 \\
#2
\end{pmatrix}}}

\newcommand{\dial}[1]{\mathbf{Dial}(#1)}

\def\lat{\operatorname{\mathbf{Lat}}}
\def\distlat{\operatorname{\mathbf{DistLat}}}

\newcommand{\latdoctrine}[2]{\xymatrix{#2 \colon #1^{\op}  \ar[r] & \lat }}

\newcommand{\distlatdoctrine}[2]{\xymatrix{#2 \colon #1^{\op}  \ar[r] & \distlat }}
\begin{document}

\lefttitle{LaTeX\ Supplement}
\righttitle{Mathematical Structures in Computer Science}

\papertitle{Article}

\jnlPage{1}{00}
\jnlDoiYr{2020}
\doival{10.1017/xxxxx}

\title{Quantifier completions, choice principles and applications}

\begin{authgrp}
\author{Davide Trotta}
\affiliation{University of Verona, Verona VR, Italy,\\
        \email{trottadavide92@gmail.com}}

\author{ Matteo Spadetto}
\affiliation{University of Leeds, Leeds, UK\\
        \email{matteo.spadetto.42@gmail.com}}
\end{authgrp}
\history{(Received xx xxx xxx; revised xx xxx xxx; accepted xx xxx xxx)}

\begin{abstract}
We contribute to the knowledge of the quantifier completions and their applications by using the language of doctrines. This algebraic presentation allows us to properly analyse the behaviour of the existential and universal quantifiers. We wish to convey the following points: the first is that these completions preserve the lattice structure and the distributive lattice structure of the fibres under opportune hypotheses which turn out to be preserved as well; the second regards the applications, in particular to the dialectica construction; the third is that these free constructions carry on some relevant choice principles.
\end{abstract}


\maketitle

\section{Introduction}

In recent years relevant logical completions involving quantifiers have been extensively studied in several areas of categorical logic. For instance,  in \cite{Hofstra2010} the author shows applications of these completions in the framework of dialectica construction, while in \cite{CPERL} are provided applications in homotopy theory. 

In \cite{ECRT} the first author introduces the notion of \emph{existential completion} of a primary doctrine, which is a free construction providing an existential doctrine starting from a primary one. This construction preserves the fragment of Regular Logic, i.e. if $\doctrine{\mC}{P}$ is an elementary and existential doctrine, then its existential completion $\doctrine{\mC}{\compex{P}}$ is again an elementary and existential doctrine. 

Originally, one of the main applications of this construction was to extend the notion of exact completion to an arbitrary elementary doctrine, but in the recent work \cite{MDM}, a strong connection between the existential completion and some choice principles is provided, e.g. doctrines which arise as instances of existential completion satisfy the Rule of Choice and every formula admits a \emph{prenex normal form} presentation. The preservation of the Regular structure, together with the validity of strong constructive principles, suggests that this construction can be applied to extend regular theories to theories satisfying the Rule of Choice, or to characterise those fragments of a given theory which satisfy the Rule of Choice, see \cite{MDM}.

Motivated by these interesting principles which hold in every doctrine arising as existential completion, we study which other logical structures are preserved by this free completion, with a focus on the Coherent fragment of logic.
First we show under which hypotheses the existential completion preserves the lattice structure of the fibres of a given doctrine $\latdoctrine{\mC}{P}$, where $\lat$ is the category of lattices. Here the crucial assumption is the existence of left-adjoints for the family of reindexings along coproduct injections. Moreover, requiring the objects of $\mC$ to be inhabited, we get the existential completion of $P$ to satisfy this property as well.

Then we show that the distributivity of meets over joins in the fibres is preserved by the existential completion as well, and this preservation depends on the assumption that the family of left-adjoints to reindexings along coproduct injections satisfies the Frobenius Reciprocity. These results allow to apply the existential completion to extend, for example, theories whose predicative part is written in the fragment of Coherent Logic or, more generally, in the fragment of Geometric Logic.

A by-product of our previous analysis regards the notion of \emph{universal completion}, which can be presented as the dual counterpart of the existential completion. In detail, observing that the universal completion $\compun{P}$ of a doctrine is nothing but the doctrine $ (-)^{\op}\compex{((-)^{\op}P)}$, where  $\freccia{\pos}{(-)^{\op}}{\pos}$ denotes the functor which inverts the order of a poset, it is the case that all our results regarding the preservation properties of the existential completion dualise in the setting of the universal completion. For instance, we automatically get when it is the case the universal completion preserves the lattice structure of the fibres and thier distributivity. Moreover, since the existential completion satisfies the Rule of Choice, we show that the universal completion of a doctrine satisfies a property which we call \emph{Counterexample Property}, which corresponds to saying that if $\forall x \psi(x)\vdash\bot$ then there exists a term $t$ such that $\psi(t)\vdash \bot$. This term $t$ represents the \emph{counterexample}. For sake of readability and thanks to the notable symmetry between these completions, we present any given result regarding the universal completion right after the corresponding one for the existential completion.

Finally we study how these constructions interact and their applications. Whenever the base category of a given universal doctrine is also cartesian closed, it is the case that its existential completion preserves the universal structure, i.e. is universal as well. In other words by sequentially applying the universal and the existential completions to a given doctrine $\latdoctrine{\mC}{P}$, we get a doctrine $\latdoctrine{\mC}{\compex{(\compun{P})}}$ which is both existential and universal. Moreover it validates a form of choice principle called AC by Troelstra in \cite{KGCW2}: $$ \forall x \exists y \alpha (x,y)\rightarrow \exists f \forall x \alpha (x, fx).$$

An relevant application of our study about preservations regards the dialectica construction, see \cite{DEPAIVA1991,HYLAND2002}. As pointed out in \cite{Hofstra2010} in the framework of fibration theory, the interaction of constructions which freely add quantifiers plays a crucial rule in the abstract presentation of the dialectica interpretation \cite{Goedel58}, in particular in the construction of the \emph{dialectica monad}. In \cite{Hofstra2010} the author shows that this monad can be decomposed as first a simple product completion, followed by simple co-product completion. The original presentation of the dialectica construction has always made it difficult to analyse the properties of the dialectica categories. This recent alternative characterisation in terms of quantifier completions, suggests that several results of preservation of these constructions might have a crucial role in spotting categorical features of diaelctica categories.

For instance, our results yield a proof that the poset reflection of the dialectica category $\dial{P}$ associated to a doctrine $\latdoctrine{\mC}{P}$ is a lattice, as well as a sufficient condition over a given doctrine that makes the poset reflection of the corresponding dialectica category into a distributive lattice. Such results give grounds for hope in further applications of this deep relation between the dialectica construction and the quantifier additions.

\medskip

\noindent
\bemph{Synopsis.} In Section 2 we recall definitions about universal and existential doctrines, and we fix the notation. In Section 3 we recall from \cite{ECRT} the existential completion, we introduce the universal one, and we prove a form a duality between the two notions. 
In Section 4 we display a set of hypotheses under which the existential and universal completions preserve the lattice structure of the fibres of a given doctrine. In Section 5 we show that the same hypotheses, together with the additional request regarding the Frobenius Reciprocity, also allow the preservation of the distributivity in the fibres. In Section 6 we present the interaction between the two constructions and the applications. In Section 7 we show which choice principles arise in doctrines which are instances of free quantifier constructions. Conclusions and future works are treated in Section 8.

\section{Brief recap on doctrines}
The notion of hyperdoctrine was introduced by F.W. Lawvere in a series
of seminal papers \cite{AF,EHCSAF}, and in the following years this notion has been revisited and generalised in several ways, see \cite{QCFF,EQC,UEC}. 

In this work we start by considering a more general notion, which we call \emph{pos-doctrine}. We specialise this notion during the next sections.
\begin{definition}
A \bemph{pos-doctrine} is a functor $\posdoctrine{\mC}{P}$, where $\mC$ is a category with finite products, and $\pos$ is the category of posets.
\end{definition}
We think of the base category $\mC$ of a pos-doctrine as the category of \emph{types} or \emph{contexts} and, for every object $A\in \mC$, we have a poset $P(A)=(P(A),\vdash)$ of \emph{predicates} on $A$.

\begin{example}\label{example subobject doct}
Let $\mC$ be a category with finite limits. The functor: $$\posdoctrine{\mC}{{\Sub_{\mC}}}$$
assigning to an object $A$ in $\mC$ the poset $\Sub_{\mC}(A)$ of subobjects of $A$ and such that for an arrow $\frecciasopra{B}{f}{A}$ the morphism $\freccia{\Sub_{\mC}(A)}{\Sub_{\mC}(f)}{\Sub_{\mC}(B)}$ is given by pulling a subobject back along $f$, is a pos-doctrine.
\end{example}
\begin{example}\label{example set-theor hyperdoctrine}
Consider the set-theoretic doctrine $\posdoctrine{\Set}{S}$. In this case $\Set$ is the category of sets and functions and, for every set $A$, $S(A)$ is the poset category of subsets of the set $A$ together with the inclusions. A functor $\freccia{S(B)}{S_f}{S(A)}$ acts as the inverse image $f^{-1}U$ on a subset $U$ of $B$.
\end{example}

\begin{example}\label{example LT}
Let $\theory{T}$ be a theory in a first order language $\lang$. We define a pos-doctrine:
$$\posdoctrine{\mC_{\theory{T}}}{LT}$$
where $\mC_{\theory{T}}$ is the category of contexts, i.e. of lists of variables and term substitutions:
\begin{itemize}
\item \bemph{objects} of $\mC_{\theory{T}}$ are finite lists of variables $\vec{x}:=(x_1,\dots,x_n)$, and we include the empty list $()$;
\item a \bemph{morphism} from $(x_1,\dots,x_n)$ into $(y_1,\dots,y_m)$ is a substitution: $$[t_1/y_1,\dots, t_m/y_m]$$ where the terms $t_i$ are built in $\lang$ on the variables $x_1,\dots, x_n$;
\item the \bemph{composition} of two morphisms $\freccia{\vec{x}}{[\vec{t}/\vec{y}]}{\vec{y}}$ and $\freccia{\vec{y}}{[\vec{s}/\vec{z}]}{\vec{z}}$ is given by substitution.
\end{itemize}
The functor $\doctrine{\mC_{\theory{T}}}{LT}$ sends a list $(x_1,\dots,x_n)$ to the set: $$LT(x_1,\dots,x_n)$$  of all well-formed formulas in the context $(x_1,\dots,x_n)$. We say that $\psi\leq \phi$, where $\phi,\psi\in LT(x_1,\dots,x_n)$, if $\psi\vdash_{\theory{T}}\phi$, and we consider its poset reflection. Given a morphism:
$$\freccia{(x_1,\dots,x_n)}{[t_1/y_1,\dots,t_m/y_m]}{(y_1,\dots,y_m)}$$ of $\mC_{\theory{T}}$, the functor $LT_{[\vec{t}/\vec{y}]}$ acts as the substitution
$LT_{[\vec{t}/\vec{y}]}(\psi(y_1,\dots,y_m))=\psi[\vec{t}/\vec{y}]$.
\end{example}

\begin{definition}\label{def existential doctrine}
A pos-doctrine $\posdoctrine{\mC}{P}$ is \bemph{existential} if, for every $A_1$ and $A_2$ in $\mC$, for any projection $\freccia{A_1\times A_2}{{\pr_i}}{A_i}$, $i=1,2$, the functor:
$$ \freccia{P(A_i)}{{P_{\pr_i}}}{P(A_1\times A_2)}$$
has a left adjoint $\Einv_{\pr_i}$, and these satisfy the \bemph{Beck-Chevalley condition:} for any pullback diagram:
$$
\quadratocomm{X'}{A'}{X}{A}{{\pr'}}{f'}{f}{{\pr}}
$$
with $\pr$ and $\pr'$ projections, for any $\beta$ in $P(X)$ the canonical arrow: 
$$ \Einv_{\pr'}P_{f'}(\beta)\leq P_f \Einv_{\pr}(\beta)$$
is an isomorphism.
\end{definition}

\begin{example}
The pos-doctrine $\posdoctrine{\mC}{\Sub}$ presented in Example \ref{example subobject doct} is existential if and only if the base category $\mC$ is regular. See \cite{QCFF}.
\end{example}
\begin{definition}\label{def universal doctrine}
A pos-doctrine $\posdoctrine{\mC}{P}$ is \bemph{universal} if, for every $A_1$ and $A_2$ in $\mC$, for any projection $\freccia{A_1\times A_2}{{\pr_i}}{A_i}$, $i=1,2$, the functor
$$ \freccia{P(A_i)}{{P_{\pr_i}}}{P(A_1\times A_2)}$$
has a right adjoint $\Ainv_{\pr_i}$, and these satisfy the Beck-Chevalley condition.
\end{definition}
\begin{example}
The pos-doctrine $\posdoctrine{\mC_{\theory{T}}}{LT}$ as defined in Example \ref{example LT} for a first order theory $\theory{T}$, is universal and existential. The right adjoints are computed by quantifying universally the variables that are not involved
in the substitution given by the projection, and similarly left adjoints are computed by quantifying existentially the variables which are not involved in the substitution given by the projection.
\end{example}

\begin{example}
The pos-doctrine $\posdoctrine{\Set}{S}$ presented in Example \ref{example set-theor hyperdoctrine} is existential and universal: on a subset $U$ of $A\x B$, for a projection $\freccia{A\x B}{\pr_A}{A}$, the right  adjoint $\Ainv_{\pr_A}$ is given by the assignment $\Ainv_{\pr_A}(U):=\{ a\in A\; |\; \pr_A^{-1}(a)\subset U \}$, and the left adjoint $\Einv_{\pr_A}$ is given by $\Einv_{\pr_A}(U):=\{a\in A\;| \exists b\in A\x B (b\in \pr_A^{-1}(a)\cap U)\}$.
\end{example}
The category of pos-doctrines $\PD$ is a 2-category, where:
\begin{itemize}
\item a \bemph{1-cell} is a pair $(F,b)$
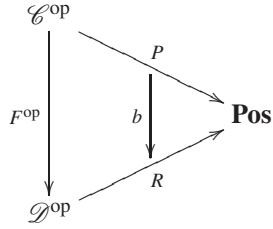
$$\duemorfismo{\mC}{P}{F}{b}{\mD}{R}$$
such that $\freccia{\mC}{F}{\mD}$ is a finite product preserving functor and: $$\freccia{P}{b}{R\circ F^{\op}}$$ is a natural transformation.
\item a \bemph{2-cell} is a natural transformation $\freccia{F}{\theta}{G}$ such that for every $A$ in $\mC$ and every $\alpha$ in $PA$, we have: 
$$b_A(\alpha)\leq R_{\theta_A}(c_A(\alpha)).$$
\end{itemize}

We denote as $\ED$ the 2-full subcategory of $\SD$ whose elements are existential pos-doctrines, and whose 1-cells are those 1-cells of $\PD$ which preserve the existential structure. Similarly, we  denote by $\UD$ the 2-full subcategory of $\SD$ whose elements are universal pos-doctrines, and whose 1-cells are those 1-cells of $\SD$ which preserve the universal structure.

\section{Quantifiers and completions}
In this section we recall from \cite{ECRT} the \emph{existential completion} and we show how this construction can be dualised to obtain a free construction which adds right adjoints along the projections, called \emph{universal completion}.

\bigskip
\textbf{Existential completion}. Let $\posdoctrine{\mC}{P}$ be a pos-doctrine. The \bemph{existential completion} $\posdoctrine{\mC}{\compex{P}}$ of $P$ is a pos-doctrine such that, for every object $A$ of $\mC$, the poset $\compex{P}(A)$ is defined as follows:
\begin{itemize}
\item \bemph{objects:} triples $(A,B,\alpha)$, where $A$ and $B$ are objects of $\mC$ and $\alpha\in P(A\x B)$.
\item \bemph{order:}  $(A,B,\alpha)\leq (A,C,\beta)$ if there exists an arrow $\freccia{A\x B}{f}{C}$ of $\mC$ such that: $$ \alpha\leq P_{\angbr{\pr_A}{f}}(\beta)$$
where $\freccia{A\x B}{\pr_A}{A}$ is the projection on $A$.
\end{itemize}
The functor $\freccia{\compex{P}(C)}{\compex{P}_f}{\compex{P}(A)}$ sends an object $(C,D,\gamma)$ of $\compex{P}(C)$ to the object $(A,D,P_{f \times 1_D}(\gamma))$ of $\compex{P}(A)$, where $\pr_A, \pr_D$ are the projections from $A\times D$.
\bigskip

The logical intuition is that an element $(A,B,\alpha)$ of the fibre $\compex{P}(A)$ represents a predicate $[a:A,b:B]\; | \;\exists b:B \phi(a,b) $.

Recall from \cite{ECRT} that the previous construction provides a free completion, i.e. it extends to a 2-functor which is left adjoint to the forgetful functor. We summarise the main properties of the existential completion in the following theorem and we refer to \cite{ECRT} for all details. Notice that this construction was originally presented for pos-doctrines factoring through the inclusion functor $\infsl \hookrightarrow \pos$. However, as observed in \cite{ECRT} itself, this notion works for general pos-doctrines as well.

\begin{therm}\label{theorem existential comp il lax-idemp}
The pos-doctrine $\compex{P}$ is existential, and the 2-monad $(\compex{\mT},\mu^{\existential}, \eta^{\existential})$ on $\SD$  is lax-idempotent. Moreover we have the isomorphism $\ED\cong \alg{\compex{\mT}}$ between the 2-category of existential pos-doctrines, and that of strict algebras for $\compex{\mT}$.
\end{therm}

\begin{remark}[Prenex normal form]\label{rem prenex normal form}
Observe that in the existential completion of a pos-doctrine $P$, every object $(A,B,\alpha)\in \compex{P}(A)$ equals: $$(A,B,\alpha
)=\compex{\Einv}_{\pr_A}\compex{\eta}_{A\x B} (\alpha).$$
\end{remark}

By dualizing the previous construction, we define the \emph{universal completion} of a pos-doctrine. Similarly to the existential case, the intuition is that an element $(A,B,\alpha)$ of the fibre $\compun{P}(A)$ of the new doctrine that we are going to define represents a predicate $ [a:A,b:B]\; | \; \forall b:B \phi(a,b) $.

\bigskip
\textbf{Universal completion}. Let $\posdoctrine{\mC}{P}$ be a pos-doctrine. The \bemph{universal completion} $\posdoctrine{\mC}{\compun{P}}$ of $P$ is a pos-doctrine such that, for every object $A$ of $\mC$, the poset $\compun{P}(A)$ is defined as follows:
\begin{itemize}
\item \bemph{objects:} triples $(A,B,\alpha)$, where $A$ and $B$ are objects of $\mC$ and $\alpha\in P(A\x B)$.
\item \bemph{order:}  $(A,B,\alpha)\leq (A,C,\beta)$ if there exists an arrow $\freccia{A\x C}{g}{B}$ of $\mC$ such that: $$ P_{\angbr{\pr_A}{g}}(\alpha)\leq \beta$$
where $\freccia{A\x C}{\pr_A}{A}$ is the projection on $A$.
\end{itemize}
Whenever $\freccia{A}{f}{C}$ is an arrow of $\mC$, the functor $\freccia{\compun{P}(C)}{\compun{P}_f}{\compun{P}(A)}$ is defined as for the existential completion.

\medskip

\begin{therm} \label{q}
The pos-doctrine $\compun{P}$ is universal.
\end{therm}

\begin{proof}
\textit{Part I. Existence of right adjoints.} Let $\posdoctrine{\mC}{P}$ be a pos-doctrine and let for every $A_1$ and $A_2$ be objects of $\mC$. The assignment $(A_1\times A_2,B,\beta)\mapsto(A_1,A_2\times B,\beta)$ defines a functor $\freccia{\compun{P}(A_1\times A_2)}{\compun{\Ainv}_{\pr_1}}{\compun{P}(A_1)}$ since, whenever $\freccia{A_1\times A_2\times C}{g}{B}$ is an arrow witnessing that: $$(A_1\times A_2,B,\beta)\leq (A_1\times A_2,C,\gamma)$$ in $P(A_1\times A_2)$, then the arrow $\freccia{A_1\times A_2\times C}{\langle \pr_{A_2}, g\rangle}{A_2 \times B}$ witnesses that $(A_1, A_2 \times B,\beta)\leq (A_1,A_2\times C,\gamma)$ in $P(A_1)$. Let us verify that $\compun{\Ainv}_{\pr_1}$ is right adjoint to $\compun{P}_{\pr_1}$. Let $(A_1,B,\beta)$ be an object of $\compun{P}(A_1)$ and let $(A_1\times A_2,C,\gamma)$ be an object of $\compun{P}(A_1\times A_2)$. If 
$$\compun{P}_{\pr_1}(A_1,B,\beta)=(A_1\times A_2,B,P_{\pr_1 \times 1_B}(\beta))\leq (A_1\times A_2,C,\gamma)$$
 then there is an arrow $\freccia{A_1\times A_2\times C}{g}{B}$ such that: 
$$P_{\langle \pr_{A_1},g\rangle}(\beta)=P_{(\pr_1 \times 1_B)\langle \pr_{A_1\times A_2},g \rangle}(\beta)=P_{\langle \pr_{A_1\times A_2},g \rangle}(P_{\pr_1 \times 1_B}(\beta))\leq \gamma.$$ 
This means that $(A_1,B,\beta) \leq (A_1,A_2\times C,\gamma)=\compun{\Ainv}_{\pr_1}(A_1\times A_2, C,\gamma)$. 

Viceversa, if the latter holds, that is, there is an arrow $\freccia{A_1\times A_2\times C}{h}{B}$ such that $P_{\langle \pr_{A_1},h \rangle}(\beta)\leq \gamma$, then: $$P_{\langle \pr_{A_1\times A_2},g \rangle}(P_{\pr_1 \times 1_B}(\gamma))=P_{\langle \pr_{A_1},h \rangle}(\beta)\leq \gamma$$ which implies that $\compun{P}_{\pr_1}(A_1,B,\beta)\leq (A_1\times A_2,C,\gamma)$.

\textit{Part II. Beck-Chevalley condition.} Let us consider the pullback square of a projection along a given arrow $f$ of $\mC$, which is of the form: $$\quadratocomm{B\times C}{B}{A \times C}{A.}{\pr_B}{f \times 1_C}{f}{\pr_A}$$ Then the corresponding: $$\quadratocomm{\compun{P}(A\times C)}{\compun{P}(A)}{\compun{P}(B\times C)}{\compun{P}(B)}{\compun{\Ainv}_{\pr_A}}{\compun{P}_{f \times 1_C}}{\compun{P}_{f}}{\compun{\Ainv}_{\pr_{B}}}$$
commutes, since it is the case that: $$\compun{P}_{f}\compun{\Ainv}_{\pr_A}(A\times C,D,\delta)=\compun{P}_{f}(A,C \times D,\delta)=(B, C\times D,P_{(f \times 1_C)\times 1_D}(\delta))$$ and that: $$\begin{aligned}\compun{\Ainv}_{\pr_{B}}\compun{P}_{f\times 1_C}(A \times C,D,\delta)&=\compun{\Ainv}_{\pr_{B}}(B \times C,D,P_{(f \times 1_C)\times 1_D}(\delta))\\&=(B,C\times D,P_{(f \times 1_C)\times 1_D}(\delta))\end{aligned}$$ whenever $(A\times C,D,\delta)$ is an object of $\compun{P}(A \times C)$.

\end{proof}
\textit{Notation}. We denote by $\freccia{\pos}{(-)^{\op}}{\pos}$ the functor which inverts the order of a poset. In  other words, if $(A,\leq_A)$ is a poset, then $(A,\leq_A)^{\op}:=(A,\leq_{A^{\op}})$ is the poset whose objects are the ones of $A$ and $a\leq_{A^{\op}} b$ if and only if $b\leq_A a$.

\begin{proposition}\label{prop un= (-)^op((-)^op P)^ex..}
Whenever $\posdoctrine{\mC}{P}$ is a pos-doctrine, it is the case that: $$\compun{P}\cong (-)^{\op}\compex{((-)^{\op}P)}.$$
\end{proposition}
\begin{proof}
Let $A$ be an object of $\mC$. The categories: $$((-)^{\op}\compex{((-)^{\op}P)})(A)=(\compex{((-)^{\op}P)}A)^{\op}\text{ and }\compex{((-)^{\op}P)}A$$ share the object class. Hence an object of the former is a triple: $$(A,B,\beta \in ((-)^{\op}P))(A\times B)=P(A\times B)^{\op}),$$ that is, a triple of the form $(A,B,\beta \in P(A\times B))$, since $P(A\times B)$ and $P(A\times B)^{\op}$ share the object class. This is nothing but an object of $\compun{P}(A)$. As the functors $((-)^{\op}\compex{((-)^{\op}P)})_{f}$ and $\compun{P}_{f}$ both send such an object to $(C,B,P_{f\times 1_B}(\beta))$ (being $\freccia{C}{f}{A}$ an arrow of $\mC$), we are left to verify that the morphism classes (that is, the ordering relations) of $(\compex{((-)^{\op}P)}A)^{\op}$ and $\compun{P}(A)$ coincide.

Let $(A,D,\delta)$ be another object of the common object class and let us assume that $(A,B,\beta)\leq (A,D,\delta)$ in $(\compex{((-)^{\op}P)}A)^{\op}$, which means that $(A,D,\delta)\leq (A,B,\beta)$ in $\compex{((-)^{\op}P)}A$. This corresponds to the existence of an arrow $\freccia{A\times D}{g}{B}$ such that: $$\delta \leq ((-)^{\op}P)_{\langle \pr_A,g \rangle}(\beta)=(P_{\langle \pr_A,g \rangle})^{\op}(\beta)=P_{\langle \pr_A,g \rangle}(\beta)$$ in $P(A\times D)^{\op}$, that is, $P_{\langle \pr_A,g \rangle}(\beta)\leq \delta$ in $P(A\times D)$. This condition is nothing but the existence of the arrow $(A,B,\beta)\leq (A,D,\delta)$ in $\compun{P}(A)$. 
\end{proof}
Notice that by to Proposition \ref{prop un= (-)^op((-)^op P)^ex..}, one can prove the analogous result of Theorem \ref{theorem existential comp il lax-idemp}.

\begin{therm}\label{theorem universal comp is colax-idem}
The universal completion extends to a 2-adjunction, and the corresponding 2-monad $(\compun{\mT},\mu^{\universal}, \eta^{\universal})$ on $\SD$ is colax-idempotent. Moreover we have the isomorphism $\UD\cong \alg{\compun{\mT}}$ between the 2-category of universal pos-doctrines, and that of strict algebras for $\compun{\mT}$.
\end{therm}

\begin{remark}[Prenex normal form]
Observe that as for the existential completion, in the universal completion of a doctrine we have that every formula admits a prenex normal form.
\end{remark}
\begin{remark}
Note that, as \cite{ECRT}, the universal and existential completions can be generalised for an arbitrary class $\Lambda$ of morphisms closed under pullbacks, compositions and which contains identities, i.e. they can be used to freely add right (and left) adjoints along the morphisms of the class $\Lambda$.
\end{remark}

\begin{remark}\label{poset-refl-dial-remark}
Recall from \cite{Hofstra2010} that the dialectica construction \cite{DEPAIVA1991,HYLAND2002} decomposes into two steps from a modern categorical perspective, following the quantifier pattern of the original translation. The dialectica category $\dial{p}$ associated to a fibration $p$, hence in particular to an indexed category (see \cite{CLTT}), is obtained by first applying the monad which freely adds simple universal quantification and then applying the monad which freely adds simple existential quantification.

Analogously, by applying the corresponding additions of quantifiers to a given pos-doctrine  $\posdoctrine{\mC}{P}$ one gets the poset reflection of the dialectica category $\dial{P}$ to coincide with the category $\compex{(\compun{P})}(1)$. 

\end{remark}

\section{Preservation of logical structures}
When a new free completion is introduced, it is natural to study and to understand which structures are preserved by this construction. 

This section is completely devoted to this analysis of the two quantifier completions presented in the previous section. Note that, by Proposition \ref{prop un= (-)^op((-)^op P)^ex..}, we can easily translate a result of preservation for the existential completion in its dual version for the universal completion. This argument will be heavily used along our presentation. 

We start by recalling a first result in this direction provided in \cite{ECRT}.

\begin{proposition}\label{p}
Let $\posdoctrine{\mC}{P}$ be a pos-doctrine such that every fibre $P(A)$ has finite meets and such that every re-indexing functor $P_f$ preserves them. Then the existential doctrine $\posdoctrine{\mC}{\compex{P}}$ has finite meets in every fibre $\compex{P}(A)$ and the functors $\compex{P}_f$ preserve them. 
\end{proposition}
By Proposition \ref{prop un= (-)^op((-)^op P)^ex..} the dual result holds for the universal completion.

\begin{proposition}\label{p'}
Let $\posdoctrine{\mC}{P}$ be a pos-doctrine such that every fibre $P(A)$ has finite joins and such that every re-indexing functor $P_f$ preserves them. Then the universal doctrine $\posdoctrine{\mC}{\compun{P}}$ has finite joins in every fibre $\compun{P}(A)$ and the functors $\compun{P}_f$ preserve them. 
\end{proposition}

The two previous propositions reflect the logical distributive property of the quantifiers over disjunctions and conjunctions. Observe that, under the hypotheses of Proposition \ref{p}, it is the case that the family of functors $\compex{\Einv}_{\pr}$, where $\pr$ is a projection, satisfies the Frobenius reciprocity (see \cite{ECRT}). The dual property holds for the universal completion.

Now we show that, under the right hypotheses, the existential completion preserves finite joins, and then, dually, that the universal completion preserves finite meets.

Recall from \cite{CARBONILACK1993} that a category with finite products and finite sums is said to be \bemph{distributive} if the canonical arrow:
$$\freccia{(A\x B)+(A\x C)}{\theta}{A\x (B+C)}$$
is an isomorphism. The canonical arrow is $\theta:= \coprodotto{\angbr{\pr_A}{j_C'\pr_C}}{\angbr{\pr_A}{j_B'\pr_B}}$, i.e. it is the unique arrow such that the diagram:
$$\xymatrix@+2pc{
A\x C \ar[dr]_{\angbr{\pr_A}{j_C'\pr_C}}\ar[r]^{j_{A\x C}\;\;\;\;\;\;\;\;} & (A\x C)+(A\x B)\ar[d]^{\theta} & A\x B \ar[l]_{\;\;\;\;\;\;\;\;j_{A\x B}}\ar[dl]^{\angbr{\pr_A}{j_B'\pr_B}}\\
& A\x (B+C)
}$$
commutes, where $j'$ are the injections in $B+C$.


To simplify the notation and the readability of the statements we introduce the following definition. 

\begin{definition}
A \bemph{lat-doctrine} is a functor $\latdoctrine{\mC}{P}$, where $\mC$ is a distributive category and $\lat$ is the category of lattices, i.e. finitely complete and finitely cocomplete posets, and finite sup\&inf-preserving maps, i.e. finite limit and finite colimit preserving functors.
\end{definition}

\begin{therm}\label{theorem  P^ex lat doc}
Let $\latdoctrine{\mC}{P}$ be a lat-doctrine such that the images through $P$ of the injections $\freccia{A}{j_A}{A+B}$ have left adjoints $\Einv_{j_A}$ which satisfy the Beck-Chevalley condition for pullbacks (when they exist) of injections which are injections themselves. Then the following properties hold.
\begin{enumerate}
\item The pos-doctrine $\compex{P}$ is a lat-doctrine.
\item Suppose that there is an arrow $\freccia{1}{c}{C}$ for every non-initial object $C$ of $\mC$ (\textnormal{$\mC$ has points}). Then the images through $\compex{P}$ of the injections: $$\freccia{A}{j_A}{A+B}$$ have left adjoints $\compex{\Einv}_{j_A}$ which satisfy the Beck-Chevalley condition for pullbacks (when they exist) of injections which are injections themselves.
\item Suppose that $\mC$ has points and that the images through $P$ of the injections $j_A$ have right adjoints $\Ainv_{j_A}$ which satisfy the Beck-Chevalley condition for pullbacks (when they exist) of injections which are injections themselves. Then the images through $\compex{P}$ of the injections $j_A$ have right adjoints $\compex{\Ainv}_{j_A}$ which satisfy the Beck-Chevalley condition for pullbacks (when they exist) of injections which are injections themselves.
\end{enumerate}
\end{therm}

\begin{proof}
\begin{enumerate}
\item Let $A$ be an object of $\mC$ and let us verify that $\compex{P}(A)$ has finite coproducts (\textit{Part I}). Moreover, let us verify that finite coproducts are preserved by the images through $\compex{P}$ of any arrow of $\mC$ of target $A$ (\textit{Part II}).

\textit{Part I. Finite coproducts in $\compex{P}(A)$.} It is the case that $(A,0,\bot_{A\times O})$ is initial in $\compex{P}(A)$ whenever $0$ is initial in $\mC$. Moreover, whenever $(A,B,\alpha)$ and $(A,C,\beta)$ are two elements of $\compex{P}(A)$, it is the case that:
$$(A,B,\alpha)\vee (A,C,\beta):=(A,B+C,P_{\theta^{-1}}(\Einv_{j_{A\x B}}(\alpha)\vee \Einv_{j_{A\x C}}(\beta)))$$ is their coproduct. In order to verify this, at first we prove that $(A,B,\alpha)\leq ((A,B+C,\Einv_{j_{A\x B}}(\alpha)\vee \Einv_{j_{A\x C}}(\beta))$. Observe that the diagram:
$$\xymatrix{
&& A\x B \ar[dd]^{\pr_A}\ar[ld]_{j_{A\x B}}\\
&(A\x B)+(A\x C)\ar[ld]_{\theta}\\
A\x (B+C) \ar[rr]_{\pr_A'}&& A
}$$
commutes by definition of $\theta$. Moreover it is the case that:
$$P_{\theta j_{A \x B}}(P_{\theta^{-1}}(\Einv_{j_{A\x B}}\alpha\vee \Einv_{j_{A\x C}}\beta))=P_{j_{A\x B}}(\Einv_{j_{A\x B}}\alpha)\vee P_{j_{A\x B}}(\Einv_{j_{A\x C}}\beta)\geq \alpha$$
since $\alpha\leq P_{j_{A\x B}}\Einv_{j_{A\x B}}(\alpha)$.
Similarly one can prove that $(A,C,\beta)\leq (A,B,\alpha)\vee (A,C,\beta)$.

Now assume that $(A,B,\alpha)\leq (A,D,\gamma)$ and $(A,C,\beta)\leq (A,D,\gamma)$, i.e. that there exist $\freccia{A\x B}{f_1}{D}$ and $\freccia{A\x C}{f_2}{D}$ such that $\alpha\leq P_{\angbr{\pr_A}{f_1}}(\gamma)$ and $\beta \leq P_{\angbr{\pr_A}{f_2}}$. By the universal property of coproducts, the diagram:
$$\xymatrix@+3pc{
A\x C \ar[rd]_{\angbr{\pr_A}{f_2}} \ar[r]^{j_{A\x C}} & (A\x C)+(A\x B) \ar[d]^{\coprodotto{\angbr{\pr_A}{f_2}}{\angbr{\pr_A}{f_1}}}& A\x B\ar[l]_{j_{A\x B}}\ar[dl]^{\angbr{\pr_A}{f_1}} \\
& A \x D
}$$
commutes. In order to simplify the notation, let $c:=\coprodotto{\angbr{\pr_A}{f_2}}{\angbr{\pr_A}{f_1}}$. First notice that the following equality holds again by the universal property of coproducts:
$$\pr_A'c=\pr_A \theta$$
where $\freccia{A\x D}{\pr_A'}{A}$ and $\freccia{A\x (B+C)}{\pr_A}{A}$. Hence the diagram:
$$\xymatrix{
&& A\x (B+ C) \ar[dd]^{\pr_A}\ar[ld]_{\theta^{-1}}\\
&(A\x B)+(A\x C)\ar[ld]_{c}\\
A\x D \ar[rr]_{\pr_A'}&& A
}$$
commutes. Now by our assumption, and since $\angbr{\pr_A}{f_1}=cj_{A\x B}$, it is the case that $\alpha\leq P_{j_{A\x B}} (P_{c}(\gamma)).$ Since $\Einv_{j_{A\x B}}\dashv P_{j_{A\x B}}$, it holds that $\Einv_{j_{A\x B}}(\alpha)\leq P_c(\gamma)$ and hence: $$P_{\theta^{-1}}(\Einv_{j_{A\x B}}(\alpha))\leq P_{\theta^{-1}}
P_c(\gamma).$$ Similarly one gets that $P_{\theta^{-1}}(\Einv_{j_{A\x C}}(\beta))\leq P_{\theta^{-1}}
P_c(\gamma)$ and therefore we can conclude that:
$$P_{\theta^{-1}}(\Einv_{j_{A\x B}}(\alpha) \vee \Einv_{j_{A\x C}}(\beta))\leq P_{\theta^{-1}}
P_c(\gamma)$$
that is, $(A,B,\alpha)\vee (A,C,\beta)\leq (A,D,\gamma)$.

\medskip

\textit{Part II. Preservation of finite coproducts.} Let $\freccia{D}{f}{A}$ be an arrow of $\mC$. Then it is the case that $\compex{P}_f(A,0,\bot_{A\times 0})=(D,0,\bot_{D\times 0})$. Moreover, one might observe that: $$P_{f\times 1_{B+C}}P_{\theta^{-1}}(\Einv_{j_{A\times B}}\alpha\wedge\Einv_{j_{A\times C}}\beta)=P_{\theta^{-1}}(\Einv_{j_{D\times B}}(P_{f\times 1_B}\alpha) \wedge \Einv_{j_{D\times C}}(P_{f\times 1_C}\beta))$$ by naturality of the class of isomorphisms: $$\freccia{X\times B + X\times C}{\theta}{X \times (B+C)},$$ where $X$ is any object of $\mC$, and by the Beck-Chevalley property of the class of left adjoints to the injections of coproducts. Observe indeed that the left-hand square of the commutative diagram: $$\xymatrix@-0.7pc{D\times B \ar[dd]_{f\times 1_B} \ar[rr]^>>>>>>>>{j_{D\times B}} && (D\times B)+(D \times C)\ar[dd]^{(f\times 1_B)+(f\times 1_C)} \ar[rr]^{\theta} && D\times (B+C) \ar[dd]^{f\times 1_{B+C}} \ar[rr]^>>>>>>>>{\pi_D} && D \ar[dd]^f \\ \\ A\times B \ar[rr]_>>>>>>>>{j_{A\times B}} && (A\times B)+(A \times C) \ar[rr]_{\theta} && A\times (B+C) \ar[rr]_>>>>>>>>{\pi_A} && A }$$ is a pullback, since the right-hand square and the outer one (whose horizontal arrows are projections from $D\times B$ and from $A\times B$) are pullbacks (and $\theta$ are isos). The same holds with with $j_{D\times C}$ and $j_{A\times C}$ instead of $j_{D\times B}$ and $j_{A\times B}$ respectively. This implies that $\compex{P}_f$ preserves binary joins.

\medskip

\item \textit{Part I. Existence of left adjoints.} Let $(A+B,C,\alpha) \in \compex{P}(A+B)$ and let $(A,D,\delta)\in \compex{P}(A)$. We define $\compex{\Einv}_{j_A}(A,D,\delta)$ to be the object $(A+B,D,P_{\theta^{-1}}\Einv_{j_{A\times D}}(\delta))$ being $\theta$ the isomorphism $(A\times D)+(B\times D)\to (A+B)\times D$ and $j_{A\times D}$ the injection $A \times D \to (A \times D)+(B\times D)$. Then the conditions $\compex{\Einv}_{j_A}(A,D,\delta)\leq (A+B,C,\alpha)$ and $(A,D,\delta)\leq \compex{P}_{j_A}(A+B,C,\alpha)$ are equivalent to the conditions:
\begin{align}\label{ecco}
&\textit{there is }\freccia{(A+B)\times D}{g}{C} \notag \\ &\textit{such that }P_{\theta^{-1}}\Einv_{j_{A\times D}}(\delta)\leq P_{\langle \pr_{A+B},g \rangle}(\alpha)
\end{align}
and:
\begin{align}\label{ecco2}
&\textit{there is }\freccia{A\times D}{h}{C} \notag \\ &\textit{such that }\delta\leq P_{\langle \pr_A,h\rangle}P_{j_A \times 1_C}(\alpha)
\end{align}
respectively. Being $\theta$ an isomorphism and by left adjointness of $\Einv_{j_{A\times D}}$, (1) is equivalent to the existence of an arrow $\freccia{(A+B)\times D}{g}{C}$ such that $\delta\leq P_{j_{A\times D}}P_{\langle\pr_{A+B}\theta,g\theta \rangle}(\alpha)$. If \eqref{ecco} holds for some $g$, then let $h:=g \theta j_{A\times D}$. Viceversa, if \eqref{ecco2} holds for some $h$, then let $g:= [h,c!]\theta^{-1}$, where $!$ is the unique arrow $B\times D \to 1$ (intuitively, we do not take care of the ‘‘$C$-values" that $g$ assumes over ‘‘$B\times D$-part of its domain", hence we might define its ‘‘restriction" to $B\times D$ as the arrow $c!$, which is the $c$-constant map $B\times D \to C$). In both cases one might verify that: $$\langle\pr_{A+B}\theta,g\theta \rangle{j_{A\times D}}={(j_A \times 1_C)}{\langle \pr_A,h\rangle}$$ which implies that \eqref{ecco} and \eqref{ecco2} are equivalent. Hence $\compex{\Einv}_{j_A}$  is left adjoint to $\compex{P}_{j_A}$ and one might as usual verify it to satisfy the Beck-Chevalley condition.

\medskip

\textit{Part II. Beck Chevalley condition.} Let us assume that the square: $$\xymatrix{C \ar[r]^{j_C} \ar[d]_f & C+D \ar[d]^{g} \\ A \ar[r]_{j_A} & A+B }$$ is a pullback and let $(A,E,\epsilon)$ be an object of $\compex{P}(A)$. We are left to prove that $\compex{\Einv}_{j_C}\compex{P}_f(A,E,\epsilon)=(C+D,E,P_{\theta^{-1}}\Einv_{j_{C\times E}}P_{f\times 1_E}\epsilon)$ and $\compex{P}_g\compex{\Einv}_{j_A}(A,E,\epsilon)=(C+D,E,P_{g\times 1_E}P_{\varphi^{-1}}\Einv_{j_{A\times E}}\epsilon)$ are equal, where $\theta$ and $\varphi$ are the isomorphisms $C\times E + D\times E \to (C+ D)\times E$ and $A\times E + B\times E \to (A+B)\times E$ respectively. Let us consider the following commutative diagram: $$\xymatrix@-0.7pc{C\times E \ar[dd]_{f\times 1_E} \ar[rr]^>>>>>>>>{j_{C\times E}} && C\times E + D\times E \ar[dd]_{\varphi^{-1}(g\times 1_E)\theta} && (C+D)\times E \ar[dd]^{g\times 1_E} \ar[ll]_{\theta^{-1}} \ar[rr]^>>>>>>>>{\pr_{C+D}} && C+D \ar[dd]^{g} \\ \\  A\times E \ar[rr]_>>>>>>>>{j_{A\times E}} && A\times E + B\times E && (A+B)\times E \ar[ll]^{\varphi^{-1}} \ar[rr]_>>>>>>>>{\pr_{A+B}} && A+B}$$ which is a pullback, since its horizontal arrow $A\times E \to A + B$ equals the arrow $j_A\pr_A$, the arrow $f\times 1_E$ is the pullback of $f$ along $\pr_A$ and $f$ is the pullback of $g$ along $j_A$ (hence $f\times 1_E$ is indeed the pullback of $g$ along $j_A\pr_A$). As the right-hand square is pullback, we deduce that the left-hand square is a pullback as well. Therefore, by Beck Chevalley condition on $\Einv$, it is the case that: $$\begin{aligned} P_{\theta^{-1}}\Einv_{j_{C\times E}}P_{f\times 1_E}\epsilon&= P_{\theta^{-1}}P_{\varphi^{-1}(g\times 1_E)\theta}\Einv_{j_{A\times E}}\epsilon \\ &=P_{g\times 1_E}P_{\varphi^{-1}}\Einv_{j_{A\times E}}\epsilon \end{aligned}$$ and we are done.

\medskip

\item Let $(A+B,C,\alpha) \in \compex{P}(A+B)$ and let $(A,D,\delta)\in \compex{P}(A)$. Then we define $\compex{\Ainv}_{j_A}(A,D,\delta)$ to be the object $(A+B,D,P_{\theta^{-1}}\Ainv_{j_{A\times D}}(\delta))$ being $\theta$ the isomorphism $(A\times D)+(B\times D)\to (A+B)\times D$ and $j_{A\times D}$ the injection $A \times D \to (A \times D)+(B\times D)$. Then the conditions $ (A+B,C,\alpha)\leq \compex{\Ainv}_{j_A}(A,D,\delta)$ and $\compex{P}_{j_A}(A+B,C,\alpha)\leq (A,D,\delta)$ are equivalent to the conditions:
\begin{align}\label{cond 1 adj destri}
&\textit{there is }\freccia{(A+B)\times C}{g}{D} \notag \\ &\textit{such that }\alpha\leq P_{\langle \pr_{A+B},g \rangle}P_{\theta^{-1}}\Ainv_{j_{A\times D}}(\delta)
\end{align}
and:
\begin{align}\label{cond 2 adj destri}
&\textit{there is }\freccia{A\times C}{h}{D} \notag \\ &\textit{such that }P_{j_A \times 1_C}(\alpha)\leq P_{\langle \pr_A,h\rangle}(\delta)
\end{align}
respectively. Observe that if \eqref{cond 1 adj destri} holds, then:
$$P_{j_A\x 1_C}(\alpha)\leq P_{\theta^{-1}\angbr{ \pr_{A+B}}{g}j_A\x 1_C}\Ainv_{j_{A\times D}}(\delta)$$
and it is the case that:
$$\theta^{-1}\angbr{\pr_{A+B}}{g}j_A\x 1_C=\theta^{-1}\angbr{j_A\pr_A}{g(j_A \x 1_C)}=j_{A\x D}\angbr{\pr_A}{g(j_A\x 1_C)}$$
because in our case $\freccia{(A+B)\x D}{\theta^{-1}}{(A\x D)+(B\x D)}$ and
$$\theta j_{A\x D}=\angbr{j_A\pr_A}{\pr_D}.$$
Moreover $P_{j_{A\x D}}\Ainv_{j_{A\x D}}\leq \id$ since $P_{j_{A\x B}}\dashv \Ainv_{j_{A\x D}} $, and then:
$$\begin{aligned}P_{j_A\x 1_C}(\alpha)&\leq P_{\theta^{-1}\angbr{ \pr_{A+B}}{g}j_A\x 1_C}\Ainv_{j_{A\times D}}(\delta)\\&=P_{\angbr{\pr_A}{g(j_A\x 1_C)}}P_{j_{A\x D}}\Ainv_{j_{A\x D}} (\delta)\\&\leq P_{\angbr{\pr_A}{g(j_A\x 1_C)}}(\delta).\end{aligned}$$
Then we can set: $$\freccia{A\x C}{h:=g(j_A\x 1_C)}{D}$$ so that 
$P_{j_A\x 1_C}(\alpha)\leq P_{\angbr{\pr_A}{h}}(\delta)$
as required in \eqref{cond 2 adj destri}.

Suppose that \eqref{cond 2 adj destri} holds. First notice that $j_A\x 1_C= \theta j_{A\x C}$, hence:
$$\alpha\leq P_{\theta^{-1}}\Ainv_{j_{A\x C}} P_{\angbr{\pr_A}{h}}(\delta).$$
Now we define $\freccia{(A+B)\x C}{g}{D}$ as the following composition of arrows:
$$\xymatrix@+1pc{
(A+B)\x C \ar[r]^{\theta^{-1}} & (A\x C)+(C\x B)\ar[r]^>>>>>>>{\id_{A\x C}+ !} & (A\x C)+1 \ar[r]^{\;\;\coprodotto{h}{d}}& D
}$$ 
where $!$ is the terminal arrow, and $\coprodotto{h}{d}$ is the coproduct of $\freccia{A\x C}{h}{D}$ and the constant $\freccia{1}{d}{D}$. Let us verify that the commutative diagram:
$$\xymatrix@+2pc{
A\x C \ar[d]_{\angbr{\pr_A}{h}}\ar[r]^{ j_{A\x C}} & (A\x C)+( B\x C)\ar[d]^{\theta^{-1}\angbr{\pr_{A+B}}{g}\theta}\\
A\x D \ar[r]_{j_{A\x D}} & (A\x D)+ (B\x D).
}$$
is a pullback, in order to use the Beck-Chevalley condition. Consider the following diagram:
$$\xymatrix@+1pc{
E \ar@/^1.8pc/[rrrd]^{\angbr{a_1}{a_2}} \ar@/_1.5pc/[ddr]_{\angbr{a_3}{a_4}} \ar@{.>}[rd]^{\angbr{a_3}{a_2}}\\
&A\x C \ar[d]_{\angbr{\pr_A}{h}}\ar[r]^>>>>>>>>{j_{A\x C}} & (A\x C)+(B\x C) \ar[r]^{\theta} & (A+B)\x C \ar[d]^{\angbr{\pr_{A+B}}{g}}\\
&A\x D \ar[r]_>>>>>>>>{j_{A\x D}} & (A\x D)+(B\x D) \ar[r]_{\theta} & (A+B)\x D.
}$$
and assume that $\angbr{\pr_{A+B}}{g}\angbr{a_1}{a_2}= j_{A\x D} \theta \angbr{a_3}{a_4}$. By the equality $\theta j_{A\x D}=\angbr{j_A\pr_A}{\pr_D}$, the following ones hold:
\begin{itemize}
\item $a_1=j_A a_3;$
\item $g\angbr{a_1}{a_2}=a_4.$
\end{itemize}
First we prove that the choice of $\angbr{a_3}{a_2}$ makes the left triangle commutes:
$$\angbr{a_3}{a_4}=\angbr{a_3}{h\angbr{a_3}{a_2}}.$$
Then we need to show that $h\angbr{a_3}{a_2}=a_4=g\angbr{a_1}{a_2}=g\angbr{j_A a_3}{a_2}$. Observe that $\angbr{j_A a_3}{a_2}=(j_A\x  1_C)\angbr{a_3}{a_2}=\theta j_{A\x C}\angbr{a_3}{a_2}$, and then, by definition of $g$, it is the case that: $$g\angbr{a_1}{a_2}=(\coprodotto{h}{d}(\id_{A\x C}+!)\theta^{-1})(\theta j_{A\x C}\angbr{a_3}{a_2})=h\angbr{a_3}{a_2}$$ where the last equality holds as the composition:
$$\xymatrix@+2pc{ A\x C \ar[r]^>>>>>>>>{j_{A\x C} }& (A\x C)+ (B\x C) \ar[r]^{\id_{A\x C}+ !} & (A\x C)+1}$$
equals the arrow $\xymatrix@+2pc{A\x C \ar[r]^{j_{A\x C}} & (A\x C)+1}$ and, by the universal property of the coproduct, it is the case that $\coprodotto{h}{d} j_{A\x C}=h.$ Hence the left triangle commutes. 

Now we consider the top triangle. This commutes because:
$$\theta j_{A\x C}\angbr{a_3}{a_2}=\angbr{j_A\pr_A}{\pr_C}\angbr{a_3}{a_2}=\angbr{j_A a_3}{a_2}$$
which is equal to $\angbr{a_1}{a_2}$ by the hypothesis $a_1=j_A a_3$.

This concludes that the previous commutative square is indeed a pullback, and then, since $\theta$ is invertible, the first square is a pullback. Therefore, by the Beck-Chevalley condition, we get that: $$\alpha\leq P_{\angbr{\pr_{A+B}}{g}}P_{\theta^{-1}} \Ainv_{j_{A\x D}}(\delta).$$ Hence $\compex{P}_{j_A}\dashv \compex{\Ainv}_{j_A}$, and one can directly check that these satisfy the Beck-Chevalley condition. The proof is analogous to \textit{Part II} of 2.
\end{enumerate}

\end{proof}

\begin{remark}[\textit{Considerations about the hypotheses of Theorem \ref{theorem  P^ex lat doc}}]
In the general hypotheses of our statement, the class of the injections of $\mC$ is not necessarily closed under pullback. However, here we require the Beck-Chevalley condition for left/right adjoints to the pullback functor along the injections to hold just for existing pullbacks of injections which are injections as well. We do so for sake of generality: this hypothesis is clearly satisfied in any concrete framework where injections are fully stable under pullback and the Beck-Chevalley condition holds w.r.t. them (e.g. extensive categories). From now on, whenever we deal with a class $\mathfrak{C}$ of arrows of $\mathcal{C}$ which are not necessarily stable under pullback and a family of functors indexed by $\mathfrak{C}$ which are respectively left (or right) adjoint to the images through $P$ of the arrows of $\mathfrak{C}$, we commit the following abuse of notation: instead of talking about \textit{the Beck-Chevalley condition for pullbacks (when they exist) of arrows of $\mathfrak{C}$ which are in $\mathfrak{C}$ themselves}, we simply write \textit{the Beck-Chevalley condition}.

Secondly, we underline the importance of the request that $\mC$ has points in the statement of Theorem \ref{theorem  P^ex lat doc} in order to prove that the functor $\compex{\Einv}_{j_A}$ is actually left adjoint to $\compex{P}_{j_A}$. If $(A+B,C,\alpha) \in \compex{P}(A+B)$ and $(A,D,\delta)\in \compex{P}(A)$ then the conditions $\compex{\Einv}_{j_A}(A,D,\delta)\leq (A+B,C,\alpha)$ and $(A,D,\delta)\leq \compex{P}_{j_A}(A+B,C,\alpha)$ (that we want to be equivalent) are witnessed by some arrows: $$\freccia{(A+B)\times D}{g}{C}\text{ and }\freccia{A\times D}{h}{C}$$ respectively, which satisfy the usual inequalities. If the former condition is witnessed by an arrow $g$, then one verifies that its restriction to $A \times D$ witnesses the latter. However, if an arrow $h$ witnesses the latter condition, then, in order to get an arrow $g$ witnessing the former, we need to arbitrarily expand it to the full domain $(A\times D)+(B\times D)$. Here is precisely where we need the existence of constant living in $D$, when $D$ is non-initial. The same consideration holds for $\compex{\Ainv}_{j_A}$.
\end{remark}

\begin{remark}\label{generalised joins}
The way we define the binary join of two objects in a fiber of the existential completion of a given doctrine $\latdoctrine{\mC}{P}$ generalises to arbitrary joins, assuming that these exist in the fibres of $P$ and that $\mC$ has arbitrary products.
\end{remark}

\begin{example}\label{esempio valeria}
Let us consider the subobject lat-doctrine $\latdoctrine{\mC}{\Sub_{\mC}}$ over a finitely complete category $\mC$ with disjoint coproducts which are stable under pullback. Observe that these are the usual assumptions over $\mC$ in \cite{DEPAIVA1991}, and they imply that $\mC$ is a distributive category. Let $A,B$ be objects of $\mC$ and let $j_A$ be the monic injection $A \to A+B$. Observe that $\Sub_{j_A}$ has the left adjoint, which is given by the functor $\freccia{\Sub(A)}{\Einv_{j_A}}{\Sub(A+B)}$ which acts as the post-composition by $j_A$.
\end{example}
Again by Proposition \ref{prop un= (-)^op((-)^op P)^ex..} the dual result of Theorem \ref{theorem  P^ex lat doc} holds for the universal completion.

\begin{therm}\label{co-theorem  P^ex lat doc}
Let $\mC$ be a distributive category and let $\latdoctrine{\mC}{P}$ be a lat-doctrine such that the images through $P$ of the injections $\freccia{A}{j_A}{A+B}$ have right adjoints $\Ainv_{j_A}$ which satisfy the Beck-Chevalley condition. Then the following properties hold.
\begin{enumerate}
\item The pos-doctrine $\compun{P}$ is a lat-doctrine.
\item Suppose that $\mC$ has points. Then the images through $\compun{P}$ of the injections: $$\freccia{A}{j_A}{A+B}$$ have right adjoints $\compun{\Ainv}_{j_A}$ which satisfy the Beck-Chevalley condition.
\item Suppose that $\mC$ has points and that the images through $P$ of the injections $j_A$ have left adjoints $\Einv_{j_A}$ which satisfy the Beck-Chevalley condition. Then the images through $\compun{P}$ of the injections $j_A$ have left adjoints $\compun{\Einv}_{j_A}$ which satisfy the Beck-Chevalley condition.
\end{enumerate}
\end{therm}

\proof Let $Q:=(-)^{\op} P$. Then $Q$ is a lat-doctrine satisfying the hypotheses of Theorem \ref{theorem  P^ex lat doc}. Hence $\compex{Q}$ satisfies the theses of Theorem \ref{theorem  P^ex lat doc}, which corresponds to saying that $\compun{P}=(-)^{\op} \compex{Q}$ (see Proposition \ref{prop un= (-)^op((-)^op P)^ex..}) satisfies the theses of the current statement.
\endproof

\begin{example}
Observe that the subobject lat-doctrine $\latdoctrine{\mC}{\Sub}$ presented in Example \ref{esempio valeria} has right adjoints along the inclusions. In particular the right adjoint $\freccia{\Sub(A)}{\Ainv_{j_A}}{\Sub(A+B)}$ of $\Sub_{j_A}$ is the functor sending a subobject represented by a mono $\freccia{S}{s}{A}$ to the subobject represented by the arrow: $$\freccia{S + B}{s + 1_B}{A + B}$$ which is indeed a mono: if $a$ be an arrow $X \to S + B$ then $X$ has a structure of coproduct together with the injections $\freccia{X_S}{a^*i_S}{X}$ and $\freccia{X_B}{a^*i_B}{X}$ obtained by pulling back along $a$ the injections: $$\freccia{S}{i_S}{S+B}\text{ and }\freccia{B}{i_B}{S+B}$$ respectively. Let us denote as $a_1$ and $a_2$ the unique arrows $X_S \to S$ and $X_B \to B$ respectively such that $a =a_1 + a_2$. Observe that the injections $\freccia{S}{i_S}{S+B}$ and $\freccia{B}{i_B}{S+B}$ are the pullbacks along $s + 1_B$ of the arrows $j_A$ and $j_B$, hence $a^*i_S$ and $a^*i_B$ are the pullbacks of $j_A$ and $j_B$ alons $(s+1_B)a$. This implies that, whenever $b$ is another arrow $X \to S + B$ such that $(s+1_B)a=(s+1_B)b$ then, by appliying the same procedure to $b$, we obtain the same coproduct structure $(a^*i_S,a^*i_B)$ over $X$. In particular $j_A s a_1 = (s+1_B)a(a^*i_S)=(s+1_B)b(a^*i_S)=j_A s b_1$, which implies that $a_1=b_1$, as $j_As$ is a monomorphism. Analogously $a_2=b_2$, hence $a=b$ and $s + 1_B$ is proven to be a monomorphism.

Let $\freccia{S}{s}{A}$ be a subobject of $A$ and let $\freccia{T}{t}{A+B}$ be a subobject of $A+B$. Then the conditions ${j_A}^*t=\Sub_{j_A}(t)\leq s$ and $t \leq \Ainv_{j_A}(s)=s+1_B$  are equivalent: if the latter holds then a mono witnessing the former exists by the universal property of the pullback; if the former holds for a mono $\freccia{{j_A}^*T}{m}{S}$ then $\freccia{T}{m+1_B}{S+B}$ witnesses the latter.
\end{example}
\section{On the distributive structure}

In the previous sections we proved that, under given hypotheses, the existential and universal completion of a lat-doctrine is a lat-doctrine as well. It is natural to ask under which hypoteses these completions preserve also the distributivity of the fibres. Essentially this depends on the Frobenius reciprocity (or on its dual version, which we call the co-Frobenius reciprocity - i.e. the one with $\Ainv$ and $\vee$ in place of $\Einv$ and $\wedge$ respectively) of the left (or right, respectively) adjoints to the pullbacks along the injections. Let us present this phenomenon in detail.

\begin{definition}
A \bemph{distributive lat-doctrine} is a functor $\distlatdoctrine{\mC}{P}$, where $\mC$ is a distributive category and $\distlat$ is the category of distributive lattices and finite co/limit-preserving (i.e. finite sup/inf-preserving) maps between them.
\end{definition}

\begin{therm} \label{Teorema 4 + frobenius}
Let $\distlatdoctrine{\mC}{P}$ be a distributive lat-doctrine such that the images through $P$ of the injections $\freccia{A}{j_A}{A+B}$ have left adjoints $\Einv_{j_A}$ which satisfy the Beck-Chevalley condition and the Frobenius reciprocity. Then the following properties hold.

\begin{enumerate}
\item The pos-doctrine $\compex{P}$ (which is a lat-doctrine by Theorem \ref{theorem  P^ex lat doc}) is a distributive lat-doctrine.

\item Suppose that $\mC$ has points. Then the left adjoints $\compex{\Einv}_{j_A}$ to the images through $\compex{P}$ of the injections $j_A$ (which exist and satisfy the Beck-Chevalley condition by Theorem \ref{theorem  P^ex lat doc}) happen to satisfy the Frobenius reciprocity.
\end{enumerate}
\end{therm}

%
\begin{proof}
\textit{Part I. Distributivity in $\compex{P}(A)$.} Let $(A,B,\beta)$, $(A,C,\gamma)$ and $(A,D,\delta)$ be objects of $\compex{P}(A)$ and let us consider the objects $((A,B,\beta)\vee (A,C,\gamma))\wedge (A,D,\delta)$ and $((A,B,\beta)\wedge (A,D,\delta)) \vee ((A,C,\gamma)\wedge (A,D,\delta))$, which are respectively $(A,(B+C)\times D,x)$ and $(A,B\times D+C\times D,y)$ whose third components $x$ and $y$ equal: $$P_{\langle \pr_A,\pr_{B+C} \rangle}P_{\theta^{-1}}(\Einv_{j_{A\times B}}\beta \vee \Einv_{j_{A\times C}}\gamma)\;)\wedge P_{\langle\pr_A,\pr_D\rangle}\delta$$ and: $$P_{\varphi^{-1}}(\;\Einv_{j_{A\times B\times D}}(P_{\langle \pr_A,\pr_B \rangle}\beta \wedge P_{\langle \pr_A,\pr_D \rangle}\delta)\vee\Einv_{j_{A\times C\times D}}(P_{\langle \pr_A,\pr_C \rangle}\gamma \wedge P_{\langle \pr_A,\pr_D \rangle}\delta)\;)$$ respectively, where $\theta$ and $\varphi$ are the isomorphisms ${(A\times B + A\times C)}\to{A\times (B+C)}$ and ${A\times B\times D+A\times C\times D}\to{A\times (B\times D+C\times D)}$ of $\mC$. As the following diagram of isomorphisms: $$\xymatrix{A\times (B+C)\times D \ar[rr]^{\theta^{-1} \times 1_D} \ar[d]_{1_A \times \omega^{-1}} && (A\times B + A\times C)\times D \ar[d]^{\psi^{-1}} \\ A \times (B\times D + C\times D) \ar[rr]_{\varphi^{-1}} && A\times B \times D + A\times C\times D}$$ commutes, in order to verify that $1_A \times \omega^{-1}$ constitutes an isomorphism $((A,B,\beta)\vee (A,C,\gamma))\wedge (A,D,\delta)\to((A,B,\beta)\wedge (A,D,\delta)) \vee ((A,C,\gamma)\wedge (A,D,\delta))$, it is enough to verify the elements: $$(\Einv_{j_{A\times B\times D}}P_{\langle \pr_A,\pr_B \rangle}\beta\vee\Einv_{j_{A\times C\times D}}P_{\langle \pr_A,\pr_C \rangle}\gamma)\wedge P_{[\langle\pr_A,\pr_D \rangle,\langle\pr_A,\pr_D \rangle]}\delta$$ and: $$\Einv_{j_{A\times B\times D}}(P_{\langle \pr_A,\pr_B \rangle}\beta \wedge P_{\langle \pr_A,\pr_D \rangle}\delta)\vee\Einv_{j_{A\times C\times D}}(P_{\langle \pr_A,\pr_C \rangle}\gamma \wedge P_{\langle \pr_A,\pr_D \rangle}\delta)$$ of $P(A\times B \times D + A \times C\times D)$ to be equal, since they are the images through $P_{(\theta \times 1_D)\psi}$ and $P_{\phi}$ respectively of the third components of the objects $((A,B,\beta)\vee (A,C,\gamma))\wedge (A,D,\delta)$ and $((A,B,\beta)\wedge (A,D,\delta)) \vee ((A,C,\gamma)\wedge (A,D,\delta))$ respectively. This is the case, by distributivity of $P(A\times B \times D + A \times C\times D)$ and being: $$\begin{aligned}&\Einv_{j_{A\times B\times D}}P_{\langle \pr_A,\pr_B \rangle}\beta \wedge P_{[\langle\pr_A,\pr_D \rangle,\langle\pr_A,\pr_D \rangle]}\delta = \\ =\;&\Einv_{j_{A\times B\times D}}(P_{\langle \pr_A,\pr_B \rangle}\beta \wedge P_{j_{A\times B\times D}}P_{[\langle\pr_A,\pr_D \rangle,\langle\pr_A,\pr_D \rangle]}\delta)= \\ =\;& \Einv_{j_{A\times B\times D}}(P_{\langle \pr_A,\pr_B \rangle}\beta \wedge P_{\langle\pr_A,\pr_D\rangle}\delta) \end{aligned}$$ by the Frobenius reciprocity on the injection: $$\freccia{A\times B \times D}{j_{A\times B \times D}}{A\times B \times D + A\times C \times D}$$ (and the same holds with $C$ in place of $B$).

\medskip

\textit{Part II. Frobenius reciprocity.} Let $(A,C,\gamma)$ and $(A+B,D,\delta)$ be objects of $\compex{P}(A)$ and $\compex{P}(A+B)$. We are left to verify the objects $(\;\compex{\Einv}_{j_A}(A,C,\gamma)\;)\wedge (A+B,D,\delta)$ and  $\compex{\Einv}_{j_A}(\;(A,C,\gamma)\wedge \compex{P}_{j_A}(A+B,D,\delta)\;)$, which equal: $$(\;A+B,C\times D, (\;P_{\langle\pr_{A+B},\pr_C \rangle}P_{\theta^{-1}}\Einv_{j_{A\times C}}\gamma\;)\wedge P_{\langle \pr_{A+B},\pr_D \rangle}\delta \; )$$ and: $$(\;A+B,C\times D,P_{\varphi^{-1}}\Einv_{j_{A\times C\times D}}(\;P_{\langle \pr_A,\pr_C \rangle}\gamma \wedge P_{\langle \pr_A,\pr_D \rangle}P_{j_A\times 1_D}\delta\;)\;)$$ respectively, to be equal, where $j_A$ is the injection $A \to A+B$ and $\theta$ and $\varphi$ are the isomorphisms $A\times C+ B\times C \to (A+B)\times C$ and $A\times C\times D +B\times C\times D\to (A+B)\times C\times D$ respectively. This is the case, since: $$\begin{aligned} &P_{\varphi^{-1}}\Einv_{j_{A\times C\times D}}(\;P_{\langle \pr_A,\pr_C \rangle}\gamma \wedge P_{\langle \pr_A,\pr_D \rangle}P_{j_A\times 1_D}\delta\;)= \\ =\;&P_{\varphi^{-1}}\Einv_{j_{A\times C\times D}}(\;P_{\langle \pr_A,\pr_C \rangle}\gamma \wedge P_{j_{A\times C\times D}}P_{\varphi}P_{\langle \pr_{A+B},\pr_D \rangle}\delta\;)= \\ \text{\{Frobenius r.\}}=\;& (\;P_{\varphi^{-1}}\Einv_{j_{A\times C\times D}}P_{\langle \pr_A,\pr_C \rangle}\gamma\;)\wedge P_{\langle \pr_{A+B},\pr_D \rangle}\delta= \\ \text{\{Beck-Chevalley c.\}}=\;& (\;P_{\varphi^{-1}}P_{[\langle \pr_A,\pr_C\rangle,\langle \pr_B,\pr_C\rangle]}\Einv_{j_{A\times C}}\gamma\;)\wedge P_{\langle \pr_{A+B},\pr_D \rangle}\delta= \\ =\;&(\;P_{\langle\pr_{A+B},\pr_C \rangle}P_{\theta^{-1}}\Einv_{j_{A\times C}}\gamma\;)\wedge P_{\langle \pr_{A+B},\pr_D \rangle}\delta \end{aligned}$$ and we are done.
\end{proof}
By Proposition \ref{prop un= (-)^op((-)^op P)^ex..}, as usual we deduce the following dual statement:

\begin{therm} \label{co - Teorema 4 + frobenius}
Let $\distlatdoctrine{\mC}{P}$ be a distributive lat-doctrine such that the images through $P$ of the injections $\freccia{A}{j_A}{A+B}$ have right adjoints $\Ainv_{j_A}$ which satisfy the Beck-Chevalley condition and the co-Frobenius reciprocity. Then the following properties hold.

\begin{enumerate}
\item The pos-doctrine $\compex{P}$ (which is a lat-doctrine by Theorem \ref{co-theorem  P^ex lat doc}) happens to be a distributive lat-doctrine.

\item Suppose that $\mC$ has points. Then the right adjoints $\compex{\Ainv}_{j_A}$ to the images through $\compex{P}$ of the injections $j_A$ (which exist and satisfy the Beck-Chevalley condition by Theorem \ref{co-theorem  P^ex lat doc}) happen to satisfy the co-Frobenius reciprocity.
\end{enumerate}
\end{therm}

Theorem \ref{Teorema 4 + frobenius} represents a specialisation of points \textit{1.} and \textit{2.} of Theorem \ref{theorem  P^ex lat doc} for distributive lat-doctrines. An analogous specialisation of point \textit{3.} is not as natural as those ones:

\begin{proposition}\label{bastarda}
Let $\distlatdoctrine{\mC}{P}$ be a distributive lat-doctrine such that the images through $P$ of the injections $\freccia{A}{j_A}{A+B}$ have left adjoints $\Einv_{j_A}$ which satisfy the Beck-Chevalley condition and the Frobenius reciprocity. Then the following property holds.

\begin{enumerate}
\item[3.] Suppose that $\mC$ has points and that the images through $P$ of the injections $j_A$ have right adjoints $\Ainv_{j_A}$ which satisfy the Beck-Chevalley condition and the co-Frobenius reciprocity. Then the right adjoints $\compex{\Ainv}_{j_A}$ to the images through $\compex{P}$ of the injections $j_A$ (which exist and satisfy the Beck-Chevalley condition by Theorem \ref{theorem  P^ex lat doc}) happen to satisfy the co-Frobenius reciprocity if and only if the following equation: \begin{equation}\label{ecco3} P_{\phi^{-1}} \Ainv_{j_{A\times (C+D)}}P_{\theta^{-1}}\Einv_{j_{A\times C}}=P_{\omega^{-1}} \Einv_{j_{(A+ B)\times C}}P_{\psi^{-1}}\Ainv_{j_{A\times C}}
\end{equation} holds for every commutative diagram of the form: $$\xymatrix@-0.5pc{A\times C+ A \times D \ar[r]^{\theta} & A \times (C+D) \ar[rr]^>>>>>>>>{j_{A\times (C+D)}} && A\times (C+D) + B \times (C+D) \ar[d]^{\phi} \\ A \times C \ar[u]^{j_{A\times C}} \ar[d]_{j_{A\times C}} & && (A+B) \times (C+D) \\ A\times C+ B \times C \ar[r]^{\psi} & (A+B)\times C \ar[rr]^>>>>>>>>{j_{(A+B)\times C}} && (A+B)\times C + (A+B) \times D. \ar[u]_{\omega}}$$
\end{enumerate}
\end{proposition}
\proof
Assuming the equality \eqref{ecco3} to hold for any commutative diagram as in the statement, one directly gets the co-Frobenius reciprocity: $$(\;\compex{\Ainv}_{j_A}(A,C,\gamma)\;)\vee (A+B,D,\delta) = \compex{\Ainv}_{j_A}(\;(A,C,\gamma)\vee \compex{P}_{j_A}(A+B,D,\delta)\;)$$ to hold for every $(A,C,\gamma)\in \compex{P}(A)$ and $(A+B,D,\delta)\in\compex{P}(A+B)$. Viceversa, if: $$(\;\compex{\Ainv}_{j_A}(A,C,\gamma)\;)\vee (A+B,D,\delta) = \compex{\Ainv}_{j_A}(\;(A,C,\gamma)\vee \compex{P}_{j_A}(A+B,D,\delta)\;)$$ for every $(A,C,\gamma)\in \compex{P}(A)$ and $(A+B,D,\delta)\in\compex{P}(A+B)$, by taking $\delta = 0_{(A+B)\times D}$ this equality collapses into the equality \eqref{ecco3} (being $\gamma$ an arbitrary element of $P(A\times C)$).
\endproof

\noindent
As usual, Theorem \ref{bastarda} dualises for the universal completion.

\noindent
We conclude the current section with the following:

\begin{remark}\label{preciao}
We observe that Theorem \ref{Teorema 4 + frobenius} tells us that the existential completion preserves type theories whose predicative part is written in the Coherent fragment of logic. More generally, by Remark \ref{generalised joins}, the same holds for type theories whose predicative part is written in the Geometric fragment of logic. See Remark \ref{ciao} for more details and examples.
\end{remark}

\section{Combining quantifier completions}
In the previous section we showed some preservation properties of the existential and the universal completion.  Here we show how to combine these two constructions.

\begin{therm}\label{theorem exist. comp preserv forall}
Let $\posdoctrine{\mC}{P}$ be a universal pos-doctrine and suppose that $\mC$ has exponents. Then $\posdoctrine{\mC}{\compex{P}}$ is existential and universal, i.e. the existential completion preserves the universal structure.
\end{therm}

\begin{proof}[Sketch of Proof. See the details in the Appendix \ref{Appendix E}]
Let $A_1,A_2$ be objects of $\mC$, and let $\freccia{A_1\x A_2}{\pr_{A_1}}{A_1}$ be the first projection. Let: $$\freccia{\compex{P}(A_1\x A_2)}{\compex{\Ainv}_{\pr_{A_1}}}{\compex{P}(A_1) }$$ be defined by: $$ (A_1\x A_2,B,\alpha)\mapsto (A_1,B^{A_2},\Ainv_{\angbr{\pr_1}{\pr_3}} P_{\angbr{\pr_1,\pr_2}{\eval \angbr{\pr_2}{\pr_3}}}(\alpha))$$ where $\pr_i$ are the projections from $A_1\x A_2 \x B^{A_2}$ and $\freccia{A_2 \x B^{A_2}}{\ev}{B}$ is the evaluation map. The intuition is that the right adjoints act by mapping a formula $\exists b:B \alpha (a_1,a_2,b)\mapsto \exists f:B^{A_2}\forall a_2:A_2 \alpha(a_1,a_2, f(a_2))$.
\end{proof}

Let $\exponent{\SD}$ be the 2-full subcategory of $\SD$ whose objects are pos-doctrines whose base category has exponents, and whose 1-cells are the 1-cells of $\SD$ which preserve the exponents. Similarly, we denote by $\exponent{\UD}$ and $\exponent{\ED}$ the two subcategories of $\exponent{\SD}$ of the universal and existential pos-doctrines whose base category has exponents.

\begin{therm}\label{theorem lifting}
The 2-monad $\freccia{\exponent{\SD}}{\compex{\mT}}{\exponent{\SD}}$ induces a 2-monad: $$\freccia{\exponent{\UD}}{\ltil{\compex{\mT}}}{\exponent{\UD}}.$$
\end{therm}
\begin{proof}[Sketch of Proof. See the details in the Appendix \ref{Appendix C}]
By Theorem \ref{theorem exist. comp preserv forall} the 2-monad $\compex{\mT}$ preserves the universal structure. A direct verification provides the right preservation of 1-cells as well as the fact that the unit and the counit preserve the universal structure.
\end{proof}

By Theorem \ref{theorem universal comp is colax-idem} it is the case that $\exponent{\UD}\cong \exponent{\alg{\compun{\mT}}}$ and then we can conclude that the existential completion induces a 2-monad
$$
\freccia{\exponent{\alg{\compun{\mT}}}}{\ovln{\compex{\mT}}}{\exponent{\alg{\compun{\mT}}}}.
$$
Observe that the 2-monad $\ovln{\compex{\mT}}$ is a \emph{lifting} of $\compex{\mT}$ on $\exponent{\alg{\compun{\mT}}}$, hence by the well-known characterization of the distributive laws of 2-monads, see for instance \cite{PDLAVB,UCTSBSSL}, we have the following result. See \cite{Hofstra2010} for the general case of arbitrary fibrations.
\begin{therm}
There exists a distributive law $\freccia{\compun{\mT}\compex{\mT}}{\lambda}{\compex{\mT}\compun{\mT}}$ and the 2-functor $\freccia{\exponent{\SD}}{\compex{\mT}\compun{\mT}}{\exponent{\SD}}$ is a 2-monad.
\end{therm}
By combining all previous statements, we finally infer the following result about existential and universal completion of lat-doctrines.
\begin{therm}\label{theorem ex un lat-doc}
Let $\latdoctrine{\mC}{P}$ be a lat-doctrine such that:
\begin{itemize}
\item the category $\mC$ has points;
\item the category $\mC$ has exponents;
\item the images $P_{j_A}$ of the injections $\freccia{A}{j_A}{A+B}$ have left and right adjoints $\Einv_{j_A}\dashv P_{j_A}\dashv \Ainv_{j_A}$, which satisfy $BC$.
\end{itemize}
Then $\latdoctrine{\mC}{\compex{(\compun{P})}}$ is an existential and universal lat-doctrine and the images $\compex{(\compun{P})}_{j_A}$ of the injection $\freccia{A}{j_A}{A+B}$ have left and right adjoints $\compex{(\compun{\Einv})}_{j_A}\dashv \compex{(\compun{P})}_{j_A}\dashv \compex{(\compun{\Ainv})}_{j_A}$, which satisfy $BC$.
\end{therm}

\proof
The statement is a consequence of Theorem \ref{theorem  P^ex lat doc}, Theorem \ref{co-theorem  P^ex lat doc} and Theorem \ref{theorem exist. comp preserv forall}.
\endproof

\begin{remark}
Let $\latdoctrine{\mC}{P}$ be a lat-doctrine which satisfies the hypotheses of Theorem \ref{theorem ex un lat-doc}. As exposed in Remark \ref{poset-refl-dial-remark}, remind that $\dial{P}$ is nothing but the fiber of $\compex{(\compun{P})}$ over the terminal object of $\mC$. Hence, by Theorem \ref{theorem ex un lat-doc}, we can conclude that the poset reflection of the dialectica category $\dial{P}$ is a lattice.
\end{remark}
The characterisation that we presented in the Proposition \ref{bastarda} of the preservation by the existential completion of the co-Frobenius reciprocity consists of a very strong condition that does not even hold in very natural enviroments, as shown in the following Example. Therefore, we conclude that the dialectica monad does not necessarily preserve the fibred-distributivity unless very restrictive and unusual conditions hold.

\begin{example}
Let us consider the usual subset-doctrine $\distlatdoctrine{\Set}{\Sub}$, which is a distributive lat-doctrine having both left and right adjoints to the pullbacks along injections which satisfy the Beck-Chevalley condition and the (co-)Frobenius reciprocity. Referring to the commutative diagram in the statement of Proposition \ref{bastarda}, whenever $S \subseteq A \times C$, it is the case that the left-hand side and the right-hand side of the equality (1) in Proposition \ref{bastarda} produce the subsets: $$S + B \times (C+D) \text{ and } S + B \times C$$ of $(A+B)\times (C+D)$ respectively. Therefore the equality (1) is not always satisfied. Hence, the doctrine $\compex{\Sub}$ is a distributive lat-doctrine that has both left and right adjoints to the pullbacks along the injections, but by Proposition \ref{bastarda} the right ones do not satisfy the co-Frobenius reciprocity anymore.

In particular, by the dual statement of Proposition \ref{bastarda}, it is the case that the lat-doctrine $\compex{(\compun{\Sub})}$ has both left and right adjoints to the pullbacks along the injections, but they do not satisfy the (co-)Frobenius reciprocities. Hence we cannot deduce that $\compex{(\compun{\Sub})}$ is distributive by applying Theorem \ref{Teorema 4 + frobenius} or Theorem \ref{co - Teorema 4 + frobenius}.
\end{example}

\section{Choice principles and quantifier completions}
The constructive features of choice principles play a fundamental role in several areas of mathematics and theoretical computer science. For example in a dependent type theory  satisfying  the  propositions  as types  correspondence  together  with  the  proofs-as-programs  paradigm, the validity of the unique choice rule or even more of the choice rule says that the extraction of a computable witness from an existential statement under hypothesis can be performed within the same theory. 

However in \cite{Maietti2017OnCR} the author shows  that  the  unique  choice  rule,  and  hence  the  choice  rule, are not valid in some important constructive theories such as Coquand’s Calculus of Constructions with indexed sum  types,  list  types  and  binary  disjoint  sums  and in  its  predicative version implemented in the intensional level of the Minimalist Foundation \cite{Maietti2017OnCR,MAIETTI2009319}.

In this section we show that quantifier completions carry on some strong choice principles, and the preservation of the logical structures suggests that this free-operation could be applied to extend a given theory to another one which satisfies the following principles.

Recall from \cite{MDM} that the existential completion of a doctrine satisfies the \emph{rule of choice}. 

\begin{therm}[Rule of Choice]\label{rem exist. comp. preserves meets}
If the fibres of a pos-doctrine $\posdoctrine{\mC}{P}$ have finite meets, then the existential completion of a primary doctrine satisfies the Rule of Choice (see \cite{MDM}), i.e.  if:
$$a:A \; |\;\top \vdash \exists b:B \; \alpha (a,b)$$
then there exists a term $ a:A\;| \; f(a):B$ called \emph{witness} such that:
$$ a:A\;|\; \top \vdash \alpha (a, f(a)).$$

\end{therm}

Again we can prove the dual property for the universal completion, and we call it \emph{counterexample property}. In this case, the logical intuition is that the \emph{witness} plays the rule of a \emph{counterexample}, i.e. if we have $\forall x:A \;\alpha(x) \vdash \bot$, then there exists a term $t$ such that $\alpha (t)\vdash \bot$. In other words, $t$ is the counterexample.
\begin{therm}[Counterexample Property]
If $\posdoctrine{\mC}{P}$ is a pos-doctrine whose fibres have finite joins, then the universal completion satisfies \emph{Counterexample Property}, i.e. if:
$$a:A \; | \; \forall b: B\; \alpha (a,b)\vdash \bot$$
then there exists a term $a:A\;|\; g(a):B$, which represents the \emph{counterexample} of the previous statement, such that:
$$a:A\; | \; \alpha (a,g(a))\vdash \bot.$$
\end{therm}
We conclude the section by proving that, under the assumption of Theorem \ref{theorem exist. comp preserv forall}, the existential completion of a universal doctrine satisfies the so called \emph{Skolemization principle}.

\begin{therm}[Skolemization]\label{theorem Skolem}
Let $\posdoctrine{\mC}{P}$ be a universal pos-doctrine with exponents, and consider its existential completion $\compex{P}$. For every $A_1,A_2$ objects of $\mC$ and every $\ovln{\alpha}=(A_1\x A_2,B,\alpha)\in \compex{P}(A_1\x A_2)$ it holds that: 
$$\compex{\Ainv}_{\pr_1}\compex{\Einv}_{\angbr{\pr_1}{\pr_2}}\eta_{A_1\x A_2\x B}(\alpha)=\compex{\Einv}_{\pr'_1}\compex{\Ainv}_{\angbr{\pr_1}{\pr_3}}\eta_{A_1\x A_2\x B^{A_2}} P_{\angbr{\pr_1,\pr_2}{\eval \angbr{\pr_2}{\pr_3}}}(\alpha)).$$

\end{therm}

\begin{proof}
Let $A_1,A_2$ be objects of $\mC$, and let $\freccia{A_1\x A_2}{\pr_{1}}{A_1}$ be the first projection. Recall that $\freccia{\compex{P}(A_1\x A_2)}{\compex{\Ainv}_{\pr_{1}}}{\compex{P}(A_1) }$ is defined by:
$$ (A_1\x A_2,B,\alpha)\mapsto (A_1,B^{A_2},\Ainv_{\angbr{\pr_1}{\pr_3}} P_{\angbr{\pr_1,\pr_2}{\eval \angbr{\pr_2}{\pr_3}}}(\alpha)).$$
By Remark \ref{rem prenex normal form}, it is the case that:
$$\compex{\Ainv}_{\pr_1}(\ovln{\alpha})=\compex{\Einv}_{\pr_1'} \eta_{A_1\x B^{A_2}}\Ainv_{\angbr{\pr_1}{\pr_3}} P_{\angbr{\pr_1,\pr_2}{\eval \angbr{\pr_2}{\pr_3}}}(\alpha).$$
By Theorem \ref{theorem lifting}, the unit of the 2-monad preserves the universal structure, hence we have that $\eta_{A_1\x B^{A_2}}\Ainv_{\angbr{\pr_1}{\pr_3}}=\compex{\Ainv}_{\angbr{\pr_1}{\pr_3}}\eta_{A_1\x A_2 \x B^{A_2}}$ and then: $$\compex{\Ainv}_{\pr_1}(\ovln{\alpha})=\compex{\Einv}_{\pr_1'} \compex{\Ainv}_{\angbr{\pr_1}{\pr_3}}\eta_{A_1\x A_2 \x B^{A_2}}P_{\angbr{\pr_1,\pr_2}{\eval \angbr{\pr_2}{\pr_3}}}(\alpha).$$ On the other hand, by Remark \ref{rem prenex normal form}, it is the case that:
$$ \compex{\Ainv}_{\pr_1}(A_1\x A_2,B,\alpha)=\compex{\Ainv}_{\pr_1}\compex{\Einv}_{\angbr{\pr_1}{\pr_2}}\eta_{A_1\x A_2\x B}(\alpha)$$
that is, the stated equality holds.
\end{proof}
The property proved in Theorem \ref{theorem Skolem} reflects the principle of \emph{Skolemization} introduced by G\"odel in his work on the dialectica interpretation \cite{Goedel58}. It is called (AC) by Toelstra in \cite{KGCW2}, since this principle is a form of choice, and it has the following presentation: $$ \forall x \exists y \alpha (x,y)\rightarrow \exists f \forall x \alpha (x, fx).$$ 

\begin{remark}\label{ciao}
Combining Theorem \ref{rem exist. comp. preserves meets} and Theorem \ref{Teorema 4 + frobenius}, we obtain an interesting tool to extend a type theory whose predicative part is written in the coherent fragment of first-order logic to another type theory which moreover satisfies the rule of choice.

More generally, we can extend a type theory whose predicative part is written in the geometric logic -e.g. the subobject-doctrine over a (quasi)topos- to another type theory which moreover satisfies the Rule of Choice (see Remark \ref{generalised joins} and Remark \ref{preciao}). Observe indeed that the subobject-doctrine over a quasitopos is both existential and elementary. Hence all the reindexing functors -and in particular the ones along the injections- have a left adjoint, so that the required hypotheses are satisfied.

A deeper analysis of the link between the existential completion and the validity of choice principles in toposes and quasitoposes is left as a future work.

\end{remark}

\section{Conclusions and further works}

We introduced and characterised the universal and existential completions, by showing those logical properties that are preserved by these constructions. We remark that the properties we needed in order to get them preserved (e.g. the existence of left and right adjoints over the injections) are all natural: they are true in many concrete instances enjoying a satisfying ‘‘power-set" algebra, like subobject doctrines over lextensive categories (see \cite{DEPAIVA1991} and Example \ref{example set-theor hyperdoctrine}).

One of the major benefits of this property-preservation analysis appears during its application to dialectica construction. During the current essay we only focused on the notion of poset-reflection of the dialectica category associated to a lattice-doctrine. In the future work we intend to generalise our approach in order to examine properties of the dialectica category itself and not just of its poset-reflection.

Moreover we are going to find the right hypotheses such that the universal and existential completion preserves the implication of intuitionistic logic, focusing on the applications. Finally, we are interested in applying our results about the quantifier completions to extend concrete theories or to characterise their fragments that satisfy the choice principles of our analysis. A comparison with the dialectica tripos and the results presented in \cite{Biering_dialecticainterpretations} is left as a future work.

\subsection*{Aknowledgements.} We would like to thank Nicola Gambino, Maria Emilia Maietti and Margherita Zorzi for their comments and suggestions.

\bibliographystyle{msclike}
\bibliography{biblio_tesi_PHD}  
\appendix

\section{Proof of Theorem \ref{theorem exist. comp preserv forall}}

\label{Appendix E}

We remind that the following argument shows a proof-irrelevant version of the corresponding result in \cite{Hofstra2010}.

\begin{proof}
\textit{Part I. Existence of right adjoints along projections.} Let $A_1,A_2$ be objects of $\mC$, and let $\freccia{A_1\x A_2}{\pr_{A_1}}{A_1}$ be the first projection. Let: $$\freccia{\compex{P}(A_1\x A_2)}{\compex{\Ainv}_{\pr_{A_1}}}{\compex{P}(A_1) }$$ be defined by: $$ (A_1\x A_2,B,\alpha)\mapsto (A_1,B^{A_2},\Ainv_{\angbr{\pr_1}{\pr_3}} P_{\angbr{\pr_1,\pr_2}{\eval \angbr{\pr_2}{\pr_3}}}(\alpha))$$ where $\pr_i$ are the projections from $A_1\x A_2 \x B^{A_2}$ and $\freccia{A_2 \x B^{A_2}}{\ev}{B}$ is the evaluation map. The intuition is that the right adjoints act by mapping a formula $\exists b:B \alpha (a_1,a_2,b)\mapsto \exists f:B^{A_2}\forall a_2:A_2 \alpha(a_1,a_2, f(a_2))$. Let us verify that $\compex{\Ainv}_{\pr_1}$ is right adjoint to $\compex{P}_{\pr_1}$.

Let $(A_1,C,\gamma)\in \compun{P}(A_1)$ and $(A_1\times A_2,B,\alpha)\in \compun{P}(A_1\times A_2)$. We are left to verify that the conditions:
$$(A_1 \times A_2,C,P_{\pr_{A_1}\times 1_C}(\gamma))=\compex{P}_{\pr_{A_1}}(A_1,C,\gamma)\leq (A_1\times A_2,B,\alpha)$$
and:
$$(A_1,C,\gamma)\leq \compex{\Ainv}_{\pr_{A_1}}(A_1\times A_2,B,\alpha)=(A_1,B^{A_2},\Ainv_{\angbr{\pr_1}{\pr_3}} P_{\angbr{\pr_1,\pr_2}{\eval \angbr{\pr_2}{\pr_3}}}(\alpha))$$
are equivalent. The former disequality is equivalent to the condition: \textit{there is an arrow} $\freccia{A_1\times A_2\times C}{g}{B}$ \textit{such that} $P_{\pr_{A_1}\times 1_C}(\gamma) \leq P_{\langle \pr_{A_1\times A_2},g\rangle}(\alpha)$, that is:
\begin{align}\label{ecco6}
&\textit{there is }\freccia{A_1\times A_2\times C}{g}{B} \notag\\ &\textit{such that }\gamma\leq \Ainv_{\pr_{A_1}\times 1_C} P_{\langle \pr_{A_1\times A_2},g\rangle}(\alpha)
\end{align}
while the latter holds precisely when:
\begin{align}\label{ecco7}
&\textit{there is }\freccia{A_1\times C}{h}{B^{A_2}} \notag\\ &\textit{such that } \gamma \leq P_{\langle \pr_{A_1},h\rangle}\Ainv_{\angbr{\pr_1}{\pr_3}} P_{\angbr{\pr_1,\pr_2}{\eval \angbr{\pr_2}{\pr_3}}}(\alpha)
\end{align}
so that we are left to prove \eqref{ecco6} and \eqref{ecco7} to be equivalent. Let $\pr'_i$ be the three projections from $A_2\times A_1 \times B^{A_2}$, $\overline{\pr}_i$ the projections from $A_1 \times A_2 \times C$ and $\overline{\pr}'_i$ the projections from $A_2\times A_1 \times C$. Moreover let us assume that \eqref{ecco6} holds for a given arrow $g$. We define the arrow $\freccia{A_1 \times C}{h}{B^{A_2}}$ to be the exponential transpose of $g\langle \overline{\pr}'_2,\overline{\pr}'_1,\overline{\pr}'_3\rangle$, hence it holds that $\eval (1_{A_2}\times h)=g\langle \overline{\pr}'_2,\overline{\pr}'_1,\overline{\pr}'_3\rangle$. We are going to prove that this choice of $h$ is such that:
\begin{align}\label{ecco8}
\Ainv_{\pr_{A_1}\times 1_C} P_{\langle \pr_{A_1\times A_2},g\rangle}=P_{\langle \pr_{A_1},h\rangle}\Ainv_{\angbr{\pr_1}{\pr_3}} P_{\angbr{\pr_1,\pr_2}{\eval \angbr{\pr_2}{\pr_3}}}
\end{align}
allowing to conclude that \eqref{ecco7} holds. Moreover observe that, from a given arrow $h$ satisfying \eqref{ecco7}, one can always recover the corresponding arrow $g$ by anti-transposing $h$ and precomposing by $\langle \overline{\pr}_2,\overline{\pr}_1,\overline{\pr}_3\rangle$ (as $\langle \overline{\pr}'_2,\overline{\pr}'_1,\overline{\pr}'_3\rangle$ is an isomorphism whose inverse is indeed $\langle \overline{\pr}_2,\overline{\pr}_1,\overline{\pr}_3\rangle$). Therefore \eqref{ecco8} would also imply that \eqref{ecco6} follows from \eqref{ecco7}, concluding our adjointness proof. Hence we are left to prove that \eqref{ecco8} holds.

Let $\varphi$ be the arrow $\freccia{A_1\times A_2 \times C}{(1_{A_2}\times h)\langle \overline{\pr}_2,\overline{\pr}_1,\overline{\pr}_3 \rangle}{A_2 \times B^{A_2}}$, and observe that the equality: $$\langle \pr_1,\pr_2,\eval \langle \pr_2,\pr_3\rangle\rangle\langle \overline{\pr}_1,\varphi \rangle =\langle \overline{\pr}_1,\overline{\pr}_2,g\rangle =\langle \pr_{A_1\times A_2},g \rangle$$ holds. Hence it is the case that $P_{\langle \pr_{A_1\times A_2},g \rangle}=P_{\langle \overline{\pr}_1,\varphi \rangle}P_{\langle \pr_1,\pr_2,\eval \langle \pr_2,\pr_3\rangle\rangle}$ and therefore (3) follows if we prove that $\Ainv_{\pr_{A_1}\times 1_C}P_{\langle \overline{\pr}_1,\phi\rangle}=P_{\langle \pr_{A_1},h\rangle}\Ainv_{\angbr{\pr_1}{\pr_3}}$ holds. By the Beck-Chevalley condition for $\Ainv$ it is enough to prove that the right-hand square of the commutative diagram: $$\xymatrix{A_2\times A_1\times C \ar[d]^{\langle \overline{\pr}'_1, \overline{\pr}'_2,h\langle \overline{\pr}'_1,\overline{\pr}'_3\rangle \rangle} \ar[rr]^{\langle \overline{\pr}'_2,\overline{\pr}'_1,\overline{\pr}'_3\rangle} && A_1\times A_2\times C \ar[rrr]^{\pr_{A_1}\times 1_C=\langle \overline{\pr}_1,\overline{\pr}_3 \rangle} \ar[d]^{\langle \overline{\pr}_1,\varphi \rangle} &&& A_1 \times C \ar[d]^{\langle \pr_{A_1},h\rangle} \\ A_2\times A_1 \times B^{A_2} \ar[rr]_{\langle \pr'_2, \pr'_1, \pr'_3\rangle} && A_1\times A_2\times B^{A_2} \ar[rrr]_{\langle \pr_1,\pr_3\rangle}   &&& A_1 \times B^{A_2} }$$ is a pullback. This is the case: the outer square is a pullback, as its horizontal arrows are the projections $A_2 \times A_1 \times C \to A_1 \times C$ and $A_2 \times A_1 \times B^{A_2} \to A_1 \times B^{A_2}$ and as $\langle \overline{\pr}'_1, \overline{\pr}'_2,h\langle \overline{\pr}'_1,\overline{\pr}'_3\rangle \rangle=1_{A_2}\times \langle \pr_{A_1},h\rangle$, and moreover the horizontal arrows of the left-hand square are isos, therefore the right-hand square is indeed a pullback as well. 

\medskip

\textit{Part II. Beck-Chevalley condition.} Let us consider a pullback of a projection along a given arrow $f$, which is of the form: 
$$\quadratocomm{D \times C}{D}{A\times C}{A.}{\pr_D}{f \times 1_C}{f}{\pr_A}$$ 
and let us verify that the corresponding equality $\compex{P}_{f}\compex{\Ainv}_{\pr_A}=\compex{\Ainv}_{\pr_D}\compex{P}_{f \times 1_C}$ holds.  Whenever $(A\times C,B,\beta) \in P(A\times C)$ we get (by applying the left and right member of the wannabe equality respectively) the elements: $$(D,B^C,P_{f\times 1_{B^C}}\Ainv_{\langle\pr_1,\pr_3 \rangle}P_{\langle\pr_1,\pr_2,\eval\langle \pr_2,\pr_3\rangle \rangle}(\beta))$$ and $$(D,B^C,\Ainv_{\langle\overline{\pr}_1,\overline{\pr}_3 \rangle}P_{\langle\overline{\pr}_1,\overline{\pr}_2,\eval\langle \overline{\pr}_2,\overline{\pr}_3\rangle \rangle}P_{f\times 1_{C\times B}}(\beta))$$ of $P(D)$, being $\pr_i$ the projections from $A\times C\times B^C$ and $\overline{\pr}_i$ the projections from $D \times C\times B^C$. We are left to prove them to be equal. By Beck-Chevally condition for $\Ainv$ it is the case that $P_{f\times 1_{B^C}}\Ainv_{\langle\pr_1,\pr_3 \rangle}=\Ainv_{\langle \overline{\pr}_1,\overline{\pr}_3\rangle}P_{f\times 1_C \times 1_{B^C}}$, hence we are left to observe that: $$P_{f\times 1_C\times 1_{B^C}}P_{\langle\pr_1,\pr_2,\eval\langle \pr_2,\pr_3\rangle \rangle}=P_{\langle\overline{\pr}_1,\overline{\pr}_2,\eval\langle \overline{\pr}_2,\overline{\pr}_3\rangle \rangle}P_{f\times 1_{C\times B}}$$ which holds because the class of arrows $$\freccia{X \times C\times B^C}{\langle\pr_1,\pr_2,\eval\langle \pr_2,\overline{\pr}_3\rangle \rangle}{X \times C\times B}$$ for $X$ in $\mC$ is a natural transformation $(-)\times C \times B^C\to (-)\times (C\times B)$.
\end{proof}

\section{Proof of Theorem \ref{theorem lifting}}

\label{Appendix C}

We divide the proof of Theorem \ref{theorem lifting} into the following Lemmas. We remind that the following argument shows a proof-irrelevant version of the corresponding results in \cite{Hofstra2010}.

\begin{lemma}\label{lemma the unit commutes with forall}
The unit of the 2-monad $\compex{\mT}$ preserves the universal structure. 
\end{lemma}
\begin{proof}
Recall from \cite{ECRT} that the unit of the 2-monad $\compex{\mT}$ is provided by the 1-cell $\freccia{P}{\compex{\eta}_P}{\compex{P}}$ of $\SD$, where $\compex{\eta}_P=(\id_{\mC},\iota)$, with $\freccia{P(A)}{\iota_A}{\compex{P}(A)}$ acting as $\alpha\mapsto (A,1, \alpha)$. Then one might check that the diagram: 
$$\xymatrix{
P(A\x B)\ar[r]^{\Ainv_{\pr_A}} \ar[d]_{\iota_{A\x B}}& P(A)\ar[d]^{\iota_A}\\
\compex{P}(A\x B) \ar[r]_{\compex{\Ainv}_{\pr_A}} & \compex{P}(A)
}$$
commutes, since $1^B\cong 1$. Therefore the unit of the 2-monad preserves the universal structure.
\end{proof}
\begin{lemma}
The multiplication of the 2-monad $\compex{\mT}$ along universal pos-doctrines with exponents preserves the universal structure.
\end{lemma}
\begin{proof}
Recall that the multiplication of the 2-monad $\freccia{(\compex{\mT})^2}{\compex{\mu}}{\compex{\mT}}$ is given by the assignment $\compex{\mu}_P=\varepsilon_{\compex{P}}$, where $\varepsilon$ is the counit of the existential completion $\freccia{\compex{P}}{\varepsilon_P}{P}$ and it is given by the pair $(\id, \zeta)$, with $\freccia{\compex{P}(A)}{\zeta_A}{P(A)}$ acting as $(A,B, \alpha)\mapsto \Einv_{\pr_1}(\alpha)$, where $\freccia{A\x B}{\pr_1}{A}$ is a projection. Then we conclude that the multiplication preserves the universal structure since, for every universal pos-doctrine $P$, the map $\freccia{(\compex{\mT})^2(P)}{\compex{\mu}_P}{\compex{\mT}(P)}$ has as domain and codomain two pos-doctrines satisfying the (AC) principle by Theorem \ref{theorem Skolem}.
\end{proof}

\begin{lemma}
Let $\freccia{P}{(F,f)}{R}$ be a 1-cell of universal pos-doctrines with exponents. Then $\freccia{\compex{P}}{\compex{\mT}(F,f)}{\compex{R}}$ preserves the universal structure.
\end{lemma}
\begin{proof}
Recall that the 1-cell $\freccia{\compex{P}}{\compex{\mT}(F,f)=(F,\ovln{f})}{\compex{R}}$ is given by $F$ itself and the functor $\freccia{\compex{P}(A)}{\ovln{f}}{\compex{R}F(A)}$ acting on the poset $\compex{P}(A)$ as:
$$(A,B,\alpha)\mapsto (FA,FB, R_{\angbr{\pr_{FA}}{\pr_{FB}}}f_{A\x B}(\alpha)).$$
Since, by hypothesis, $f$ is natural with respect to $\Ainv^P$ and by definition of $\compex{\Ainv}$, one might check, by using the naturality of $f$, that the 1-cell of $(F,\ovln{f})$ preserves the universal structure.
\end{proof}

\end{document}